\documentclass[12pt,reqno]{amsart}
\usepackage{geometry,color,xcolor}                
\usepackage{graphicx, ulem}
\usepackage{amssymb}
\usepackage{epstopdf}
\usepackage{mathabx}
\usepackage{ stmaryrd }
\usepackage{esint}

\newcommand \Sd{{\mathbb{S}^2}}

\newcommand{\sm}{\setminus}

\newcommand{\sq}{\subseteq}
\newcommand{\eps}{\epsilon}

\newcommand \Om{\Omega}
\newcommand \om{\omega}
\newcommand \cp{\operatorname{cap}}
\newcommand \N{\mathbb{N}}

\title{Spherical caps do not always maximize Neumann eigenvalues on the sphere}

\author[D. Bucur]
{Dorin Bucur}
\address[Dorin Bucur]{Univ. Savoie Mont Blanc, CNRS, LAMA \\
73000 Chamb\'ery, France}
\email[D. Bucur]{dorin.bucur@univ-savoie.fr}

\author[R. Laugesen]{Richard S. Laugesen}
\address[Richard Laugesen]{Department of Mathematics, University of Illinois, Urbana,
IL 61801, U.S.A.}
\email{Laugesen@illinois.edu}

\author[E. Martinet]{Eloi Martinet}
\address[Eloi Martinet]{Institut für Mathematik, JMU Würzburg \\ 97074 Würzburg, Germany}
\email{eloi.martinet@uni-wuerzburg.de}

\author[M. Nahon]{Micka\"{e}l Nahon}
\address[Micka\"{e}l Nahon]{ Univ. Grenoble Alpes, CNRS, Grenoble INP, LJK, 38000 Grenoble, France.}
\email{mickael.nahon@univ-grenoble-alpes.fr}


\usepackage{doi}

\newcommand{\e}{\varepsilon}
\newcommand{\R}{{\mathbb R}}

\newtheorem{theorem}{Theorem}
\newtheorem{proposition}[theorem]{Proposition}
\newtheorem{lemma}[theorem]{Lemma}
\newtheorem{remark}[theorem]{Remark}

\begin{document}
\begin{abstract}
We prove the existence of an open set $\Om\subset\mathbb{S}^2$ for which the first positive eigenvalue of the Laplacian with Neumann boundary condition exceeds that of the geodesic disk having the same area. This example  holds for large areas  and  contrasts with results by Bandle and later authors proving maximality of the disk under additional topological or geometric conditions, thereby revealing such conditions to be necessary.
\end{abstract}

\maketitle

\section{Introduction}
On a smooth subdomain $\Om$ of the standard unit sphere $\Sd$ in $\R^3$, the Laplace--Beltrami operator with Neumann boundary conditions has discrete spectrum consisting of eigenvalues, denoted
$$0= \mu_0(\Om) <  \mu_1(\Om) \le \dots\le \mu_k(\Om) \to +\infty$$
and counted with respect to multiplicity. The eigenfunction corresponding to $\mu_0=0$ is constant. The first and second positive eigenvalues are  
\[
\mu_1(\Om) = \min_{ u \in H^1(\Om) \sm \{0\}, \int_\Om u =0 } \frac{\int_\Om |\nabla u|^2 }{\int_\Om u^2 } , \qquad \mu_2(\Om) = \min_{S\in{\mathcal S}_{2}} \max_{u \in S\sm \{0\}} \frac{\int_\Om |\nabla u|^2 }{\int_\Om u^2 },
\]
where integrals are with respect to surface area measure and ${\mathcal S}_2$ is the family of $2$-dimensional subspaces of 
$\{u\in H^1(\Om):\int_{\Om}u=0\}$. The eigenfunctions $(u_k)_{k\geq 0}$ satisfy 
\[
\begin{cases}
-\Delta_{\Sd} u_k = \mu_k(\Om) u_k &\mbox { in } \Om,\\
\partial_\nu u_k = 0  &\mbox { on } \partial \Om,
\end{cases} 
\]
where $\partial_\nu$ is the inward normal derivative. 

The question of maximizing $\mu_1(\Om)$ under an area constraint $|\Om|=m$ for given ``mass'' $m \in (0,4\pi)$ has been studied intensively in the last fifty years, and yet intriguing open problems remain unsolved. 

To briefly summarize this history, Ashbaugh and Benguria \cite{AB95} proved that if $m \le 2\pi$ and $\Om$ is contained in a hemisphere then
\begin{equation}\label{blmn01}
\mu_1(\Om) \le \mu_1(B_{\theta_m})
\end{equation}
where $B_{\theta_m}$ is a spherical cap of area $m$ and $\theta_m$ is the corresponding aperture angle or geodesic radius, measured from the center of the cap. Whether this hemisphere inclusion condition is merely technical or instead indicates a genuine mathematical obstruction has remained an open question, stated for example in \cite[Conjecture 1.3]{LL22}. A step toward understanding this point was taken by Bucur, Martinet and Nahon \cite{BMN22}, who relaxed the hemisphere condition to inclusion in the complement of a spherical cap of mass $m$. The smaller $m$ is, the larger the region in which $\Om$ may be placed. However, in the same paper, the authors numerically exhibited a density or weight whose first eigenvalue is larger than that of the spherical cap having the same mass, thus leading to the conclusion that even if inequality \eqref{blmn01} holds true without the inclusion restriction, the way to prove it cannot follow existing mass transplantation arguments. Further, for domains having area somewhat greater than $2\pi$, Martinet \cite{M23} found an example whose first eigenvalue, computed numerically, exceeds that of the spherical cap having the same mass. This counterexample domain is not simply connected and indeed has rather complicated topology. 

Results by Bandle \cite{Ba72} preceded those of Ashbaugh and Benguria, although with different hypotheses. She proved that inequality \eqref{blmn01} holds provided one restricts to simply connected domains $\Om$ having at most half the area of the sphere, $|\Om|\le 0.50 |\Sd|$. This approach imposes no restriction on the parts of the sphere the set is permitted to occupy. Langford and Laugesen \cite{LL22} recently improved from 50\% to 94\% of the sphere, proving \eqref{blmn01} for simply connected domains with area $|\Om|\le 0.94 |\Sd|$. 

The main purpose of this paper is to prove analytically that inequality \eqref{blmn01} can fail for multiply connected domains, by constructing a rigorous counterexample whose topology is relatively simple. Precisely, for a  large  mass $m$, e.g. above  $0.81|\Sd|$,  we prove that there exists a non-simply connected set of mass $m$ whose first eigenvalue is larger than that of the corresponding spherical cap. 

Write $\Theta$ for the unique aperture defined by the relation $\mu_1(B_{\Theta})\sin^2 \Theta =1$; it is shown in \cite[Propositions 3.1 and 4.2]{LL22} that $\mu_1(B_\theta)\sin^2 \theta >1$ on $(0,\Theta)$ and $\mu_1(B_\theta)\sin^2 \theta <1$ on $(\Theta,\pi)$ and that the value of $\Theta$ is approximately $0.7\pi$:
\[
0.70\pi<\Theta<0.71\pi .
\]

The counterexample to \eqref{blmn01} is provided by:
\begin{theorem}\label{main_th}
There exists $\delta>0$ such that for each $\theta\in (\Theta-\delta,\pi)$, there exists an open set $\Om_\theta \sq \Sd$ with $|\Om_\theta|=|B_\theta|$ and 
$$\mu_1(\Om_\theta) >\mu_1 (B_{\theta}).$$
\end{theorem}
 By no means do we claim optimality of the lower bound $\Theta-\delta$. On the contrary, we provide numerical evidence supporting the idea that such a counterexample can exist for lower areas, although still above $2\pi$.  Apertures in the range $(\Theta,\pi)$ correspond to the area or mass $m$ lying between about $0.81 |\Sd|$ and $|\Sd|$. 

The key point underlying the theorem is that when $\theta>\Theta$, the geodesic ball $B_\theta$ possesses hot spots, meaning interior local extrema of the first eigenfunctions. The strategy for proving the theorem is to start with such a spherical cap of aperture $\theta$ and remove four small ellipses of large eccentricity that are symmetrically placed at the latitude of the hot spots (see Figure \ref{fig:example_domain}). The theorem then follows from an asymptotic analysis of the first  (simple)  eigenvalue of this modified set, building on work of Ozawa \cite{Oz83},  however with the extra difficulty of dealing with a geometry where the first eigenvalue is multiple.

\begin{figure}
  \centering
  \label{fig:example_domain}
  \includegraphics[width=0.3\textwidth]{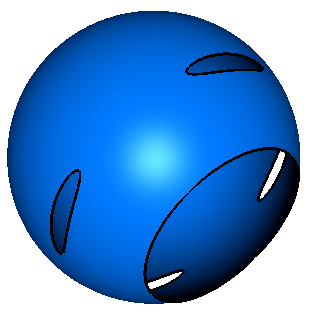}
  \includegraphics[width=0.3\textwidth]{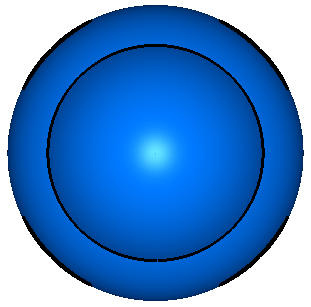}
  \includegraphics[width=0.3\textwidth]{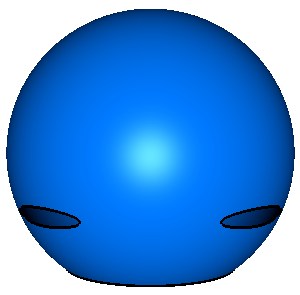}
  \caption{Different views of one punctured domain used in the proof of Theorem \ref{main_th}.}
\end{figure}

A different kind of counterexample is pursued in Section \ref{sec_num}, where we provide a non-rigorous argument for ``helmet'' or ``basket'' counterexamples, as depicted in Figure \ref{fig:sphere_strips}. We support our construction also with numerical evidence. The approach suggests that the lower bound $\Theta-\delta$ in Theorem \ref{main_th} might be significantly lowered. 

\subsection*{Open problems} Some significant problems remain to be resolved. 
\begin{itemize}
\item When $\theta\leq\frac{1}{2}\pi$, is it true that the first positive eigenvalue $\mu_1$ is maximal on the spherical cap $B_{\theta}$ among domains of measure $|B_\theta|$, with no further restriction on the domain?
\item Among simply connected domains, is there a counterexample when $|B_\theta|>0.94|\Sd|$?
\end{itemize} 

\subsection*{Overview} The paper is structured as follows. In the next section we recall known facts about the first and second eigenfunctions of spherical caps. In Section \ref{sec_conv} we give a precise asymptotic for Neumann eigenvalues in the presence of several holes, following in the footsteps of Ozawa \cite{Oz83}. Section \ref{blmn02} estimates the second term of the asymptotic development of $\mu_1$ on the perturbed spherical cap, and hence proves the main theorem. In Section \ref{sec_num} we present the non-rigorous argument for helmet counterexamples.  Section \ref{sec_asymptotic} proves that our rigorous method  cannot provide  counterexamples for simply connected sets  having mass close to $4\pi$. This is done  by comparing the first term asymptotic of the whole sphere minus a single hole with the capacity of that hole. Finally, Section \ref{sec:lem_estansatz} establishes various technical estimates used in proving the main theorem. 

\section{Structure of the first and second eigenfunctions on the sphere}
Recall that $\mu_0=0$. The separated form of the eigenspace for $\mu_1(\theta)(:=\mu_1(B_{\theta}))$ of the spherical cap $B_\theta$ is given by the next proposition, which is drawn from \cite[Proposition 4.1]{LL22}. Parts (a) and (b) are distinguished by whether the aperture is smaller or larger than the value $\Theta \simeq 0.7\pi$ that was specified earlier. 

Throughout the paper we adopt spherical coordinates, parametrizing the sphere $\mathbb{S}^2$ by
\[
(t,\phi)\in [0,\pi]\times \R / 2\pi \mathbb{Z} \ \mapsto \ (\sin t \cos \phi,\sin t \sin \phi , \cos t ) .
\]
\begin{proposition} \label{basicsphere} 
If $0<\theta<\pi$ then the eigenspace of $-\Delta_{\Sd} \, u = \mu_1(\theta) u$ on the spherical cap $B_\theta$ with Neumann boundary condition is spanned by two functions of the form 
\[
g(t) \cos \phi \qquad \text{and} \qquad g(t) \sin \phi ,
\]
where the radial part $g$ is smooth and satisfies the Neumann boundary condition $g^\prime(\theta)=0$, the differential equation 
\begin{equation} \label{gequation}
-\frac{1}{\sin t}\left((\sin t)g'(t)\right)'+\frac{1}{\sin^2 t }g(t)=\mu_1(B_\theta)g(t)
\end{equation}
and has $0=g(0)<g(t)$ whenever $t \in (0,\theta]$.

(a) If $0 < \theta \leq \Theta$ then $g(t)$ is strictly increasing for $t \in (0,\theta)$, with $g'>0$ there.

(b) If $\Theta<\theta<\pi$ then $g(t)$ first strictly increases and then strictly decreases; more precisely, a maximum point $t_{max}(\theta)\in (\frac{\pi}{2},\Theta)$ exists such that $g'>0$ on $(0,t_{max}(\theta))$ and $g'<0$ on $(t_{max}(\theta),\theta)$.
\end{proposition}
The proposition is illustrated in \cite[Figure 2]{LL22}, which we recommend consulting.

\section{  Multiple Neumann eigenvalues on perturbed domains  }\label{sec_conv}

In this section we develop some convergence results for eigenfunctions and  multiple eigenvalues of the Neumann Laplacian on singularly perturbed domains, specifically, on domains  enjoying simultaneous perforation  with  several  small holes and boundary variation. The goal of this section is to establish Proposition \ref{Prop_limholes}, which gives a precise asymptotic for the eigenvalues  in such a situation. 

 The literature is rich in these type of estimates. For instance, we refer   to \cite{MNP00,NS08,FLO23,La12,Ka20} and references therein. However, for the convenience of the reader and since we did not find a reference where all our specific needs are  met, we prefer to produce the computation in a self-contained manner,  based on the approach of Nazarov and Sokolowski \cite{NS08}. 

\subsection{Statement of the result}

Throughout this section, we fix $M$ to be a compact closed $2$-dimensional Riemannian manifold and take an open subset $\Om \subseteq M$ with smooth boundary (which allows also for the case $\partial\Om=\emptyset$). Consistent with the notation used in the introduction for subdomains of the sphere, we write $\mu_i(\Om)$ for the Neumann eigenvalues of the Laplace--Beltrami operator on $\Om$, starting with $\mu_0(\Om)=0$. Let $(u_i)_{i \in \N}$ be a corresponding choice of $L^2$-normalized eigenfunctions. Take: 
\begin{itemize}
\item $N(\geq 1)$ distinct points  $O_1,\hdots,O_N \in \Om$;
\item for each $n=1,\hdots,N$, a conformal local chart $\psi_n$ from a neighborhood of $0$ (say $B_r$ for some small radius $r>0$) in $\R^2$ into $\Om$ with $\psi_n(0)=O_n$. We suppose $r$ is small enough that the sets $(\psi_n(B_r))_{n=1}^N$ are disjoint to each other and do not intersect the boundary. Denote by $\rho_n:B_r\to \R_+$ the conformal factor associated to the pullback of the metric of $M$ through $\psi_n$ (meaning, the pullback of the metric of $M$ by $\psi_n$ is $\rho_n g_0$, where $g_0$ is the euclidean metric on $B_r$). 
\item for each $n=1,\hdots,N$, a smoothly bounded closed set $\om_n$ in $\R^2$ whose interior contains the origin, $0\in\; \stackrel{\circ}{\om}_n$.
\item a bounded integer interval $I\subset \N^*$ such that 
\[
\mu_{\min(I)-1}(\Om)<\mu_{\min(I)}(\Om) = \mu_{\max(I)}(\Om)<\mu_{\max(I)+1}(\Om) .
\]
We denote by $\mu$ the common value of the $\mu_i(\Om)$ for $i\in I$. In particular, note that $\mu>0$.
\item a smooth vector field $\phi$ on $M$ that vanishes outside a small neighborhood of $\partial\Om$ which does not intersect the sets $\psi_n(B_r)$. Write $(\phi^t)_{t\in\R}$ for its associated flow.
\end{itemize}

We define
\[
\Om^\eps=\phi^{\eps^2}\left(\Om\setminus \bigcup_{n=1}^{N}\psi_n(\eps\om_n)\right) .
\]

We will denote by $\nu$ a normal vector on $\partial \Om^\eps$ with the following convention: it points outward of $\psi_n(\eps\om_n)$ (and so points into $\Om^\eps$), and on the exterior boundary $\partial\Om^\eps \setminus \cup_n \psi_n(\partial(\eps\om_n))=\partial\phi^{\eps^2}(\Om)$ it points outward from $\Om^\eps$. Then Stokes' theorem gives for arbitrary vector fields $f$ that $\int_{\Om}\nabla\cdot f=\int_{\partial\phi^{\eps^2}(\Om)} \langle f , \nu \rangle - \sum_{n=1}^{N}\int_{\partial\psi_n(\eps\om_n)}\langle f,\nu\rangle$.\smallbreak

We consider two exterior problems, for each $\om_n$.
\begin{itemize}
\item[1) ]There exists a unique smooth vector field $W_n:\R^2\setminus \om_n\to\R^2$ such that
\[
\begin{cases}
\Delta W_n=0 & (\text{in\ }\R^2\setminus \om_n)\\
W_n(x)=\mathcal{O}_{|x|\to\infty}\left(\frac{1}{|x|}\right)\\
\partial_\nu W_n(x)=-\nu(x)& (\text{on\ }\partial\om_n)
\end{cases}
\]
where $\nu$ is the exterior normal on $\partial \om_n$. This $W_n$ is easily shown to exist; it tends to a finite limit at infinity (no logarithmic term) since $\int_{\partial\om_n} \partial_\nu W_n = - \int_{\partial\om_n}\nu=0$. After subtracting a suitable constant vector, one obtains that $W_n$ vanishes at infinity: $W_n(x)=\mathcal{O}_{|x|\to\infty}(1/|x|)$. 

For $p,q \in \{1,2\}$, let
\[
(M_n)_{p,q}=|\om_n|\delta_{p,q}+\int_{\R^2\setminus{\om_n}} \nabla (W_n)_p \cdot\nabla (W_n)_q
\]
where $(W_n)_p$ is the $p$-th component of $W_n$. Note $M_n$ is a $2\times 2$ symmetric matrix that, as a sum of two positive matrices, is itself positive. Further, $M_n$ determines the behavior of $W_n$ near infinity:
\begin{equation} \label{W_n_M_n}
W_n(x)=\frac{M_n x}{2\pi |x|^2}+\mathcal{O}_{|x|\to \infty} \! \left(\frac{1}{|x|^2}\right) 
\end{equation}
by \cite[formula (15)]{NS08}; we give a proof of this fact in Section \ref{sec:lem_estansatz}. (Incidentally, $M_n$ is linked to the so-called \textit{virtual mass}, which appears in \cite[Note G]{P51}.) 

\item[2]) Similarly, there exists a unique $2\times 2$ matrix-valued function $L_n:\R^2\setminus\om_n\to M_2(\R)$ such that
\[
\begin{cases}
\Delta L_n=0 & (\text{in\ }\R^2\setminus\om_n)\\
\partial_\nu L_n(x)=x\otimes \nu(x) & (\text{on\ }\partial\om_n)\\
L_n(x)=\left(\frac{|\om_n|}{2\pi}\log|x|\right) \! I_2+\mathcal{O}_{|x|\to\infty}\left(\frac{1}{|x|}\right) .
\end{cases}
\]
This time the existence proof uses the divergence theorem to show that 
\[
\int_{\partial\om_n} \partial_\nu L_n(x) = |\om| I_2 = \int_{\partial B_R} \partial_\nu \left(\frac{|\om_n|}{2\pi}\log|x|\right) \! I_2 
\]
for each large $R>0$. 
\end{itemize}

When these exterior problems are defined from a domain $\om$, we may also denote them by $W_\om$ (associated to the virtual mass matrix $M_\om$) and $L_\om$.

The next proposition quantifies how the eigenvalues of $\Om^\eps$ perturb away from the eigenvalues of $\Om$. The formula involves the vector field $\phi$ that was chosen above on a neighborhood of $\partial \Om$. 
\begin{proposition}\label{Prop_limholes}
If the $(\kappa_i)_{i\in I}$ are the (increasingly ordered) eigenvalues of the real symmetric matrix
\[
\left(\int_{\partial\Om}\left(\langle\nabla u_i,\nabla u_j\rangle-\mu u_i u_j\right)\langle\phi,\nu\rangle+\sum_{n=1}^{N}\left(\rho_n(0)\mu |\om_n| u_i(O_n)u_j(O_n)-\langle M_n\nabla (u_i\circ\psi_n)(0),\nabla (u_j\circ\psi_n)(0)\rangle\right)\right)_{\! i,j\in I\times I}
\]
then for some $\eps_0>0$ we have
\[
\left|\mu_i(\Om^\eps)-\mu-\eps^2\kappa_i\right|\leq (\text{const.}) \eps^{2+1/4} 
\]
for all $\eps\in (0,\eps_0)$. 
\end{proposition}

In the case where $\om$ is simply connected and smooth, the matrix $M_\om$ (associated to the domain $\om$) may be computed as follows: write  $D$ for the unit disk in $\R^2 \simeq {\mathbb C}$ and denote by $T_\om:\R^2\setminus D \to \R^2\setminus \om$ the unique biholomorphic conformal mapping that maps $\infty$ to $\infty$ with the normalization
\[
T_\om(z)=\cp[\om]\left(z+\sum_{k\geq 0}c_k[\om]z^{-k}\right)
\]
where $\cp[\om]$ is the logarithmic capacity of $\om$. Let $e_\theta=\begin{pmatrix}\cos \theta \\ \sin \theta \end{pmatrix}$. Then
\begin{equation}\label{eq_computation_M}
\langle M_\om e_\theta,e_\theta\rangle=2\pi\cp[\om]^2\left(1-\text{Re}(c_1[\om]e^{-2i\theta})\right)
\end{equation}
by using $\varphi=W \cdot e_\theta$ and $\vec{h}=e_\theta$ in \cite[eq. (11) p. 244]{P51}. In particular, $\text{Tr}(M_\om)=4\pi\cp[\om]^2$.

\subsection{Definition of the ansatz}\label{subsec_ansatz}
First we choose a $L^2(\Om)$-orthonormal basis of the eigenspace associated to $\mu$ (denoted $(u_i)_{i\in I}$) such that the symmetric matrix of Proposition \ref{Prop_limholes} is diagonal, with the $(i,i)$ coefficient being $\kappa_{i}$.\\

We let $\zeta\in\mathcal{C}^\infty_c(B_r,\R)$ be a function that is equal to $1$ in a neighborhood of $0$. We then let $F_{n,i}:\R^2\setminus\{0\}\to\R$ be the function defined by
\begin{equation}\label{def_Fni}
F_{n,i}(x)=(\rho_n^{-1}\Delta+\mu)\left[\zeta(x)\left(\frac{\langle x,M_n\nabla (u_i\circ\psi_n)(0)\rangle}{2\pi|x|^2}+\frac{\rho_n(0)\mu |\om_n|u_i(O_n)}{2\pi}\log|x|\right)\right] .
\end{equation}

Since $x/|x|^2$ and $\log |x|$ are harmonic on $\R^2\setminus\{0\}$, their Laplacians vanish. Thus the singularity of $F_{n,i}(x)$ near the origin arises only from the factor of $\mu$ times $x/|x|^2$ and $\log |x|$, and so $F_{n,i}$ is dominated there by $|x/|x|^2|=1/|x|$. Hence $F_{n,i} \in L^p(\R^2)$ for all $p\in [1,2)$. 

Let $\Delta_t$ be the laplacian of the metric $g_t$ obtained by pulling back the metric of $M$ through the flow $\phi^{t}$, and $\nu_t$ be the exterior normal vector of $\partial \Om$ for this same metric (meaning $\nu_t=(\phi^t)^*\nu_{\partial\phi^t(\Om)}$ where $\nu_{\partial \phi^t(\Om)}$ is the exterior normal vector of $\partial \phi^t(\Om)$), and denote
\[
Kf:=\lim_{t\to 0}\frac{\Delta_t f-\Delta f}{t},\qquad \eta=\lim_{t\to 0}\frac{\nu_t -\nu}{t}\text{ on }\partial\Om.
\]
Notice $Kf$ vanishes outside the neighborhood of $\partial \Om$ on which the flow acts nontrivially.

We assert that for every $i$, the following equation defines a unique function $v_i \in H^1(\Om)$:
\begin{equation}
\begin{split}
-\Delta v_i -\mu v_i & = \kappa_i u_i+Ku_i+\sum_{n=1}^{N}F_{n,i}\circ\psi_n^{-1} \qquad (\text{in }\Om) , \\
\partial_\nu v_i & = -\langle \eta, \nabla u_i\rangle \qquad (\text{on }\partial\Om) .\\
\int_{\Om}v_iu_j&=0,\ \forall j\in I
\end{split} \label{vdef}
\end{equation}
Indeed, to justify the solvability it suffices to observe that the righthand side of the equation in $\Om$ belongs to $L^p$ for some $p>1$ (in fact, for all $p\in (1,2)$) and then to show for each $j\in I$ that  
\[\int_{\Om}\left(\kappa_i u_i+Ku_i+\sum_{n=1}^{N}F_{n,i}\circ\psi_n^{-1}\right)u_j=-\int_{\partial\Om} (\partial_\nu v_i) u_j.\]
This last assertion is equivalent to
\[\kappa_i\delta_{i,j}=-\int_{\Om}\left(Ku_i+\sum_{n=1}^{N}F_{n,i}\circ\psi_n^{-1}\right)u_j+\int_{\partial\Om}\langle\eta,\nabla u_i\rangle \, u_j.\]
It may be checked that the right-hand side above is exactly the matrix of Proposition \ref{Prop_limholes}, and therefore this equation can be verified by our definitions of the $(\kappa_i)_{i\in I}$: we do not give the full computation, but instead point out that
\begin{equation}\label{formula_shapederivative}
-\int_{\Om}(Ku_i)u_j+\int_{\partial\Om}\langle \eta , \nabla u_i\rangle \, u_j=\int_{\partial\Om}\left(\langle \nabla u_i,\nabla u_j \rangle -\mu u_i u_j\right)\langle \phi,\nu\rangle
\end{equation}
by shape derivative computations with Neumann boundary conditions (see Section \ref{sec:lem_estansatz} for details that are essentially based on the computations in Fall and Weth \cite{Fall}) and
\begin{equation} \label{eq_Fni}
-\int_{\Om}(F_{n,i}\circ\psi_n^{-1})u_j=\rho_n(0)\mu |\om_n| u_i(O_n)u_j(O_n)-\langle M_n\nabla (u_i\circ\psi_n)(0),\nabla (u_j\circ\psi_n)(0)\rangle
\end{equation}
by an application of Green's theorem (for details, see Section \ref{sec:lem_estansatz}).

We may then define the ansatz: for $i \in I$ and $\eps>0$, let 
\begin{equation}\label{eq_ansatz}
\begin{split}
\widehat{u}_{i,\eps}(x)&=u_i(x)+\eps^2 v_i(x)+\eps\sum_{n=1}^{N}\zeta(\psi_n^{-1}(x)) \langle W_n(\psi_n^{-1}(x)/\eps),\nabla (u_i\circ\psi_n)(0)\rangle\\
&+\eps^2\sum_{n=1}^{N}\zeta(\psi_n^{-1}(x))  \left( - L_n(\psi_n^{-1}(x)/\eps):D^2 (u_i\circ\psi_n)(0)+\frac{\rho_n(0)\mu|\om_n|u_i(O_n)}{2\pi}\log \eps\right)\\
& \text{ for }x\in\Om\setminus\cup_{n=1}^{N}\psi_n(\eps\om_n),\text{ where $D^2$ is the Hessian operator and }A:B \overset{\text{def}}{=}\sum_{k,\ell=1}^2 A_{k,\ell}B_{k,\ell} , \\
u_{i,\eps}&=\widehat{u}_{i,\eps}\circ\phi^{-\eps^2}\text{ (defined on }\Om^\eps\text{)},\\
\mu_{i,\eps}&=\mu+\eps^2\kappa_i .
\end{split}
\end{equation}
(In the definition of $\widehat{u}_{i,\eps}(x)$, one takes $\zeta(\psi_n^{-1}(x))=0$ when $x \notin \psi_n(B_r)$.)

\begin{lemma}\label{lem_estansatz}
There exists some constant $C>0$ such that for every small enough $\eps>0$ and every $z\in H^1(\Om^\eps)$ with unit norm, we have
\[\int_{\Om^\eps}\left(\mu_{i,\eps}u_{i,\eps}z-\langle \nabla u_{i,\eps},\nabla z \rangle \right)\leq C\eps^{2+\frac{1}{2}}.\]
\end{lemma}

We defer the proof of this lemma to Section \ref{sec:lem_estansatz}.

\subsection{Proof of Proposition \ref{Prop_limholes}}

For the proof it is more convenient to work with a compact operator on $\Om^\eps$ rather than an unbounded one. Denote by 
\[\langle u,v\rangle_{H^1(\Om^\eps)}=\int_{\Om^\eps}\left( \langle \nabla u , \nabla v \rangle +u v\right)\]
the inner product that defines the norm on $H^1(\Om^\eps)$.

We first define, for a domain $V\subset M$, the operator $S_V:L^2(M)\to L^2(M)$ such that for any $f\in L^2(M)$ we have $S_V f=0$ in $M\setminus V$ and $S_Vf|_{V}\in H^1(V)$, and
\[\begin{cases}
S_Vf-\Delta (S_Vf)=f & (\text{in }V)\\
\partial_{\nu_V}S_V f=0 & (\text{on }\partial V)
\end{cases}\]
and such that the associated eigenvalues of $S_V$ are exactly the numbers $\left(\frac{1}{1+\mu_i(V)}\right)_{\! i\in\N}$ (and $0$). 

These eigenvalues are all bounded by $1$ and so the operator norm is likewise bounded: $\lVert S_V \rVert \leq 1$.

We claim  $S_{\Om^\eps}$ converges to $S_\Om$ in the $L^2(M)$-operator norm as $\eps \to 0$. Indeed  $\Om^\eps$ (and its complement) converges in the Hausdorff sense as $\eps \to 0$ to $\Om\setminus\{O_1,\hdots,O_N\}$ (and its complement), which is equivalent to $\Om$ with regard to the Laplace problem. Moreover, the set
\begin{equation} \label{uniformextension}
\bigcup_{0<\eps<1}\{u\in L^2(M):\text{Spt}(u)\subset\Om^\eps,\ \Vert u
\Vert_{H^1(\Om^\eps)}\leq 1\}
\end{equation}
is compact in $L^2(M)$, since there is a uniform $H^1$-extension operator to a neighborhood of $\bigcup_{0<\eps<1}\Om^\eps$ (achieved by extending harmonically into the holes using Lemma \ref{lem_extension_holes} in Section \ref{sec:lem_estansatz}, and taking a standard $H^1$-extension across the outer boundary of $\Om^\eps$ to obtain an extended function with compact support in the neighborhood). To prove the operator norm convergence, consider a bounded sequence $(f_\eps)_\eps$ in $L^2(M)$. It is sufficient to prove that $\Vert S_{\Om}f_\eps-S_{\Om^\eps}f_\eps\Vert_{L^2(M)}\to 0$. In particular, it suffices to do so along every subsequence. From any subsequence we may extract another subsequence (denoted still $f_\eps$, for simplicity) such that $f_\eps$ converges $L^2(M)$-weakly to some $f\in L^2(M)$. Since $S_\Om$ is compact (by Rellich's theorem), $S_\Om f_\eps$ converges in $L^2(M)$ to its limit $v=S_\Om f$ that is characterized by the weak form of the equation:
\begin{equation} \label{weak1}
\langle \nabla v,\nabla \varphi\rangle_{L^2(M)}+\langle v,\varphi\rangle_{L^2(M)}=\langle f\chi_{\Om},\varphi\rangle_{L^2(M)}, \qquad \varphi\in\mathcal{C}^\infty(M) ,
\end{equation}
where $\nabla v=0$ outside $\Om$ by convention. Now let $v_\eps=S_{\Om^\eps}f_\eps$, so that $\Vert v_\eps\Vert_{H^1(\Om^\eps)}$ is bounded independently of $\eps$. The compactness in \eqref{uniformextension} above yields a subsequence $v_{\eps_k}$ that converges in $L^2(M)$ to some $w\in L^2(M)$. Similarly, $\nabla v_{\eps_k}$ (which we recall is extended by $0$ outside $\Om^{\eps_k}$) is bounded in $L^2(M)$, and so some subsequence converges $L^2(M)$-weakly to a limiting vector field $G\in L^2(M)$. By the Hausdorff convergence of the domains, $w$ and $G$ are supported in $\Om$. Using test functions supported in $\Om$, one finds $w$ is weakly differentiable there with $\nabla w=G$. Next, the partial differential equation satisfied by $v_{\eps_k}$ says in its weak form that  
\[
 \langle \nabla v_{\eps_k},\nabla \varphi\rangle_{L^2(M)}+\langle v_{\eps_k},\varphi\rangle_{L^2(M)}=\langle f\chi_{\Om^ {\eps_k}},\varphi\rangle_{L^2(M)} , \qquad \varphi\in\mathcal{C}^\infty(M) .
\]
Passing to the limit as $\eps_k\to 0$ implies that
\begin{equation} \label{weak2}
 \langle \nabla w,\nabla \varphi\rangle_{L^2(M)}+\langle w,\varphi\rangle_{L^2(M)}=\langle f\chi_{\Om},\varphi\rangle_{L^2(M)} , \qquad \varphi\in\mathcal{C}^\infty(M) .
\end{equation}
Thus the functions $v$ in \eqref{weak1} and $w$ in \eqref{weak2} satisfy the same variational characterization, so that $w=v$ by uniqueness. Hence  $S_{\Om^{\eps_k}}f_{\eps_k} = v_{\eps_k} \to w = v = S_\Om f$. That concludes the proof that $S_{\Om^\eps}\to S_\Om$ in the operator norm.

As a consequence, for every $i$ we have 
\begin{equation} \label{convergencezerothorder}
\mu_i(\Om^\eps)\overset{\eps\to 0}{\longrightarrow}\mu_i(\Om) .
\end{equation}

Now define a compact resolvent operator $R_\eps:H^1(\Om^\eps)\to H^1(\Om^\eps)$ by
\begin{equation} \label{resolventdefn}
\forall u,v\in H^1(\Om^\eps): \quad \langle R_\eps u,v\rangle_{H^1(\Om^\eps)}=\langle u,v\rangle_{L^2(\Om^\eps)} .
\end{equation}
In other words, $R_\eps u$ verifies $R_\eps u-\Delta R_\eps u=u$ in $\Om^\eps$ and $\partial_\nu R_\eps u=0$ on $\partial\Om^\eps$.

Let $(\mu_i(\Om^\eps),u_i(\Om^\eps))_{i\in \N}$ be the eigenvalues and eigenfunctions of $-\Delta$ on $\Om^\eps$ with Neumann conditions, normalized in $L^2(\Om^\eps)$. Then the normalized eigenpairs associated to $R_\eps$ are exactly 
\[\left(\frac{1}{1+\mu_i(\Om^\eps)},\frac{u_i(\Om^\eps)}{\Vert u_i(\Om^\eps)\Vert_{H^1(\Om^\eps)}}\right).\]
where the eigenvalues are this time enumerated in a decreasing way. For an interval $A\subset\R$ and $u\in H^1(\Om^\eps)$, we define the projection operator
\[P_\eps^{A}u=\sum_{i:(1+\mu_i(\Om^\eps))^{-1}\in A}\left\langle u,\frac{u_i(\Om^\eps)}{\Vert u_i(\Om^\eps)\Vert_{H^1(\Om^\eps)}}\right\rangle_{\!\! H^1(\Om^\eps)} \frac{u_i(\Om^\eps)}{\Vert u_i(\Om^\eps)\Vert_{H^1(\Om^\eps)}}.\]

Let now $\lambda_{i,\eps}=\frac{1}{1+\mu_{i,\eps}}$, $U_{i,\eps}=\frac{u_{i,\eps}}{\Vert u_{i,\eps}\Vert_{H^1(\Om^\eps)}}$. Note that $\Vert u_{i,\eps}\Vert_{H^1(\Om^\eps)} = \sqrt{\mu_i(\Om^\eps)+1} \overset{\eps\to 0}{\longrightarrow} \sqrt{\mu+1}$. Then
\begin{align}
\Vert R_\eps U_{i,\eps}-\lambda_{i,\eps} U_{i,\eps}\Vert_{H^1(\Om^\eps)}&=\frac{1}{(1+\mu_{i,\eps})\Vert u_{i,\eps}\Vert_{H^1(\Om^\eps)}}\Vert (1+\mu_{i,\eps})R_\eps u_{i,\eps}-u_{i,\eps}\Vert_{H^1(\Om^\eps)} \notag \\
&=\mathcal{O}(1)\sup_{\Vert z\Vert_{H^1(\Om^\eps)}\leq 1}\langle (1+\mu_{i,\eps})R_\eps u_{i,\eps},z\rangle_{H^1(\Om^\eps)}-\langle u_{i,\eps},z\rangle_{H^1(\Om^\eps)} \notag \\
&=\mathcal{O}(1)\sup_{\Vert z\Vert_{H^1(\Om^\eps)}\leq 1}\langle (1+\mu_{i,\eps})u_{i,\eps},z\rangle_{L^2(\Om^\eps)}-\langle u_{i,\eps},z\rangle_{H^1(\Om^\eps)} \qquad \text{by \eqref{resolventdefn}} \notag \\
&=\mathcal{O}(1)\sup_{\Vert z\Vert_{H^1(\Om^\eps)}\leq 1}\int_{\Om^\eps}\left(\mu_{i,\eps} u_{i,\eps} z- \langle \nabla u_{i,\eps} , \nabla z \rangle \right) \notag \\
&\leq C\eps^{2+\frac{1}{2}} \label{fivehalves}
\end{align}
by Lemma \ref{lem_estansatz}. Consequently, defining the interval $A_{i,\eps}=[\lambda_{i,\eps}-\eps^{2+\frac{1}{4}},\lambda_{i,\eps}+\eps^{2+\frac{1}{4}}]$, we have
\begin{align}
& \left\Vert U_{i,\eps}-P_\eps^{A_{i,\eps}}U_{i,\eps}\right\Vert_{H^1(\Om^\eps)}^2 \notag \\
&=\sum_{j:(1+\mu_j(\Om^\eps))^{-1}\notin A_{i,\eps}}\left\langle U_{i,\eps},\frac{u_j(\Om^\eps)}{\Vert u_j(\Om^\eps)\Vert_{H^1(\Om^\eps)}}\right\rangle_{\!\! H^1(\Om^\eps)}^{\!\! 2} \notag \\
& \leq\eps^{-4-\frac{1}{2}} \sum_{j:(1+\mu_j(\Om^\eps))^{-1}\notin A_{i,\eps}}\left((1+\mu_j(\Om^\eps))^{-1}-\lambda_{i,\eps}\right)^2\left\langle U_{i,\eps},\frac{u_j(\Om^\eps)}{\Vert u_j(\Om^\eps)\Vert_{H^1(\Om^\eps)}}\right\rangle_{\!\! H^1(\Om^\eps)}^{\!\! 2} \notag 
\\
&\leq \eps^{-4-\frac{1}{2}}\Vert (R_\eps -\lambda_{i,\eps})U_{i,\eps}\Vert_{H^1(\Om^\eps)}^2 \notag \\
&\leq C\eps^\frac{1}{2} \label{eps14}
\end{align}
by squaring \eqref{fivehalves}. 

When $\eps<C^{-2}$, this inequality already gives that $P_\eps^{A_{i,\eps}}U_{i,\eps}\neq 0$, and so $A_{i,\eps}$ contains some eigenvalue of $\Om^\eps$. To finish proving the proposition, it is necessary to be a bit more attentive about the multiplicity of the eigenvalues $(\kappa_i)_{i\in I}$. By the convergence of the eigenvalues in \eqref{convergencezerothorder}, there exists $\delta>0$ such that for all sufficiently small $\eps>0$, 
\[
\mu_i(\Om^\eps)\in [\mu-\delta,\mu+\delta]\text{ if and only if }i\in I.
\]
This insures that the eigenvalues we obtain in the intervals $A_{i,\eps}$ have the correct index (belonging to $I$). We partition $I$ into $\bigcup_{r=1}^{s} I_r$ such that $i,j$ are in the same subinterval of the partition if and only if $\kappa_i=\kappa_j$. Let $B_{r,\eps}=A_{i,\eps}$ for any $i\in I_r$. Then for small enough $\eps$, the intervals $(B_{r,\eps})_{r=1,\hdots,s}$ are disjoint. We also argue that for $\eps$ small enough, the functions
\[\left(P^{B_{r,\eps}}_\eps U_{i,\eps}\right)_{i\in I_r}, \qquad r=1,\ldots,s,\]
are linearly independent. Independence is clear between different $r$-values due to disjointness of the $B_{r,\eps}$. For functions with the same $r$-value, the matrix $\left(\langle P^{B_{r,\eps}}_\eps U_{i,\eps},P^{B_{r,\eps}}_\eps U_{j,\eps}\rangle_{H^1(\Om^\eps)}\right)_{i,j\in I_r}$ of inner products converges by \eqref{eps14} to the same limit (as $\eps \to 0$) as the matrix $\left(\langle  U_{i,\eps}, U_{j,\eps}\rangle_{H^1(\Om^\eps)}\right)_{i,j\in I_r}$, which converges to the identity. 

Since the total number of independent eigenfunctions on $\Om^\eps$ associated to eigenvalues in the interval $[\mu-\delta,\mu+\delta]$ is exactly $|I|$, and since  
\[
(1+\mu_i(\Om^\eps))^{-1}\in A_{i,\eps} \quad \Longrightarrow \quad |\mu_i(\Om^\eps) - (\mu+\eps^2 \kappa_i)| \leq C \eps^{2+\frac{1}{4}} ,
\]
we see the proof of Proposition \ref{Prop_limholes} is complete.

\section{Proof of the main result, Theorem \ref{main_th}}\label{blmn02}
Let us first treat the case $\theta \in (\Theta,\pi)$. Take $t=t_{max}(\theta)\in (\pi/2,\Theta)$ as in Proposition \ref{basicsphere}(b), noting that $g^\prime(t)=0$ and hence $\mu_1(B_t)=\mu_1(B_\theta)$. We consider the four equally spaced points $x_0, x_1, x_2, x_3$ at latitude $t$ on the unit sphere with spherical coordinates 
\[
\left(t , \frac{j\pi}{2} \right), \quad j=0,1,2,3.
\]
Let $\lambda=\lambda(\theta)>1$ be a parameter that will later be fixed close to $1$, define 
\[
T(z)=\frac{\lambda z+\lambda^{-1}z^{-1}}{\sqrt{\lambda^2-\lambda^{-2}}} , 
\]
and consider the domain $\om=\R^2\setminus T(\R^2\setminus D)$, where $D$ is the unit disk in $\R^2 \simeq {\mathbb C}$. In other words $\om$ is an ellipse of area $\pi$ whose horizontal and vertical semi-axes are respectively
\[
\sqrt{\frac{\lambda+\lambda^{-1}}{\lambda-\lambda^{-1}}} \qquad \text{and} \qquad \sqrt{\frac{\lambda-\lambda^{-1}}{\lambda+\lambda^{-1}}} .
\]
Now take a small number $\eps>0$, let $\psi_j:\R^2\to \Sd$ be the stereographic projection such that $\psi_j(0)=x_j$ and $d\psi(0)e_2$ is a unit vector pointing upward (in the direction of the line of longitude) in the tangent plane to the sphere, and let $\Om=B_\theta$ and 
\[
\Om^\eps:=B_{\theta(\eps)}\setminus\bigcup_{j=0}^{3}\psi_j(\eps \om),
\]
where $\theta(\eps)$ is chosen such that $|\Om^\eps|=|B_{\theta(\eps)}|$: this means that $\theta(\eps)=\theta+\frac{4\pi \eps^2}{|\partial B_\theta|}+o(\eps^2)$. Thus $\Om^\eps$ is a spherical cap with four (approximately) elliptical holes removed near aperture $t$, with the longer axis of each hole being (approximately) horizontal. 

Let us now compute, for a fixed $\lambda$, the asymptotic of $\mu_1(\Om^\eps)$ as $\eps\to 0$. We denote by 
\begin{equation}\label{eq_descriptionu1u2}
u_1(\cdot,\phi)=  g(\cdot) \cos \phi \qquad \text{and} \qquad u_2(\cdot,\phi)=   g(\cdot) \sin \phi
\end{equation}
the eigenfunctions corresponding to the eigenvalue $\mu=\mu_1(B_\theta)=\mu_2(B_\theta)$, where $g$ is normalized such that $u_1,u_2$ are orthonormal in $L^2(B_\theta)$. Since $\Om$ is invariant by rotation of the sphere by 90 degrees around the vertical axis and the eigenfunctions associated to $\mu_1(\Om^\eps)$ and $\mu_2(\Om^\eps)$ are close (in the $L^2$ sense) to linear combinations of $u_1,u_2$ that are not invariant under the rotation, it follows that $\mu_1(\Om^\eps)=\mu_2(\Om^\eps)$.

Proposition \ref{Prop_limholes} with $N=4$ and $I=\{1,2\}$ gives the convergence of $\mu_1(\Om^\eps)$ to $\mu_1(B_\theta)$, with 
\[
\begin{split}
\mu_1(\Om^\eps)-\mu_1(B_\theta)
=\frac{1}{2}\sum_{i=1}^{2}\left[\mu_i(\Om^\eps)-\mu\right] 
& = \frac{\eps^2}{2} (\kappa_1+\kappa_2)+o(\eps^2)\\
& = \frac{\eps^2}{2} \big( \text{trace of the matrix in Proposition \ref{Prop_limholes}} \big) +o(\eps^2) .
\end{split}
\]
Evaluating the diagonal entries of the matrix in Proposition \ref{Prop_limholes} gives that 
\begin{equation}\label{est_mu1}
\begin{split}
& \mu_1(\Om^\eps)-\mu_1(B_\theta) \\
&=\frac{\eps^2}{2}\sum_{i=1}^{2}\sum_{j=0}^{3}\left[\mu_i(B_\theta)|\om||u_i(x_j)|^2- \langle M_\om  \psi_j^*\nabla u_i(x_j), \psi_j^*\nabla u_i(x_j)\rangle\right]\\
& \qquad + \frac{1}{2}\frac{4\pi \eps^2}{|\partial B_\theta|}\sum_{i=1}^{2}\int_{\partial B_\theta}\left(|\nabla u_i|^2-\mu_i(B_\theta)|u_i|^2\right)+o(\eps^2)\\
&=2\pi\eps^2\left[\mu_1(B_\theta)g(t)^2-\frac{(M_\om)_{1,1}}{\pi}\frac{g(t)^2}{\sin^2 t }-\frac{(M_\om)_{2,2}}{\pi}g'(t)^2+\left(\frac{1}{\sin^2 \theta}-\mu_1(B_\theta)\right) g(\theta)^2 \right] +o(\eps^2) ,
\end{split}
\end{equation}
as we now explain in detail. Based on the form of $u_1$ and $u_2$ in equation \eqref{eq_descriptionu1u2}, one has $u_1(\theta,\phi)^2+u_2(\theta,\phi)^2=g(\theta)^2$ and 
\[
|\nabla u_1(\theta,\phi)|^2+|\nabla u_2(\theta,\phi)|^2=g'(\theta)^2+\frac{g(\theta)^2}{\sin(\theta)^2}=\frac{g(\theta)^2}{\sin(\theta)^2}
\]
since $g'(\theta)=0$. For fixed $j$, the values taken by $u_1 \circ \psi_j$ and $u_2 \circ \psi_j$ at the origin are $\{\pm g(t),0\}$ and the gradients are $\{\pm g'(t)e_2,\pm \frac{g(t)}{\sin(t)}e_1\}$. Here each $\pm$ sign depends on $j$, but fortunately, the choice of sign makes no difference to the following calculation. Hence 
\[
\begin{split}
& \sum_{i=1}^{2}\left[\mu_i(B_\theta)|\om||u_i(x_j)|^2- \langle M_\om  \psi_j^*\nabla u_i(x_j), \psi_j^*\nabla u_i(x_j)\rangle\right] \\
& = \mu_1(B_\theta)|\om|g(t)^2-(M_\om)_{1,1} \, \frac{g(t)^2}{\sin^2 t }-(M_\om)_{2,2} \, g'(t)^2 .
\end{split}
\]
Summing over $j=0,1,2,3$ and recalling that $|\om|=\pi$ yields the final line of \eqref{est_mu1}. 

The third and fourth terms in \eqref{est_mu1} vanish because our choice $t=t_{max}(\theta)$ gives $g^\prime(t)=0$. For the fifth term, since $\theta>\Theta$ we have $\frac{1}{\sin^2 \theta}-\mu_1(B_\theta)>0$. For the second term, we now  compute $(M_\om)_{1,1}$ using formula \eqref{eq_computation_M}. In this case 
\begin{align*}
\cp[\om]&=\frac{\lambda}{\sqrt{\lambda^2-\lambda^{-2}}} , \qquad c_1[\om]=\lambda^{-2}, \\
(M_\om)_{1,1}&=\frac{2\pi \lambda^2}{\lambda^2-\lambda^{-2}}\left(1-\frac{1}{\lambda^2}\right)=\frac{2\pi}{1+\lambda^{-2}} .
\end{align*}

Coming back to \eqref{est_mu1}, we obtain
\begin{align*}
\mu_1(\Om^\eps)-\mu_1(B_\theta)&\geq 2\pi\eps^2\left[\mu_1(B_\theta)g(t)^2-\frac{1}{\pi}\frac{2\pi}{1+\lambda^{-2}}\frac{g(t)^2}{\sin^2 t }\right]+o(\eps^2)\\
&=2\pi\eps^2\frac{g(t)^2}{\sin(t)^2}\left[\mu_1(B_\theta)\sin^2 t -\frac{2}{1+\lambda^{-2}}\right]+o(\eps^2) .
\end{align*}

Our choice of $t$ ensures that $\mu_1(B_\theta)=\mu_1(B_t)$, and  since $t<\Theta$ we know  that $\mu_1(B_t)\sin^2 t >1$. Thus we may choose $\lambda$ close enough to $1$ that $\mu_1(B_\theta)\sin^2 t > \frac{2}{1+\lambda^{-2}}$. Notice this choice of $\lambda$ depends only on $\theta$. Hence for all small enough $\eps$, we have $\mu_1(\Om^\eps)>\mu_1(B_\theta)$.   This concludes the proof of the theorem in the case $\theta>\Theta$.

\begin{remark} \rm
If we sought instead counterexamples to the maximality of the geodesic ball for the arithmetic mean $(\mu_1+\mu_2)/2$ or  harmonic mean $2/\big( \frac{1}{\mu_1}+\frac{1}{\mu_2} \big)$, then it would suffice to employ a single elliptical hole instead of four holes as used above. Piercing by multiple holes is done simply to ensure the double multiplicity of $\mu_1$. Thus while the harmonic mean is maximal at a geodesic ball among simply connected domains of fixed area smaller than $0.94|\Sd|$, by \cite{LL22}, for each area larger than $|B_{\Theta}|(\approx 0.8|\Sd|)$ a doubly connected domain exists whose harmonic mean of $\mu_1$ and $\mu_2$ is larger than for the geodesic ball of the same area. \qed
\end{remark}

The remaining case in the theorem, for $\theta\in (\Theta-\delta,\Theta]$, follows from the next proposition. Recall that $\om=\om_\lambda$ is an ellipse of area $\pi$. As above, $\Om^\eps=B_{\theta(\eps)}\setminus\bigcup_{j=0}^{3}\psi_j\left(\eps \om\right)$ is a cap with four holes such that $|\Om^\eps|=|B_{\theta}|$.
\begin{proposition} \label{pr:nearTheta}
There exists $\delta>0, \lambda>1$ and $t < \Theta-\delta$ such that if $\theta \in (\Theta-\delta,\Theta]$ then $\mu_1(\Om^\eps)>\mu_1(B_{\theta(\eps)})$ for all small $\eps > 0$.
\end{proposition}
\begin{proof} 
Let $t \in(0,\Theta)$ and compute the asymptotic of $\mu_1(\Om^\eps)-\mu_1(B_{\theta})$ as follows. With the same computation as previously, applying Proposition \ref{Prop_limholes} to $\Om^\eps$ gives, following equation \eqref{est_mu1}: 
\begin{align*}
& \mu_1(\Om^\eps)-\mu_1(B_{\theta}) \\
& = 2\pi\eps^2\left(\mu_1(B_\theta)g(t)^2-\frac{(M_{\om_\lambda})_{1,1}}{\pi}\frac{g(t)^2}{\sin^2 t }-\frac{(M_{\om_\lambda})_{2,2}}{\pi}g'(t)^2 + \left( 1 - \mu_1(B_\theta)\sin^2 \theta\right) \frac{g(\theta)^2}{\sin^2 \theta} \right)+o(\eps^2) .
\end{align*}
The factor in parentheses multiplying $2\pi\eps^2$ is continuous with respect to $\theta,t,\lambda$, using here that the radial part $g$ of the eigenfunction on $B_\theta$ varies continuously with the radius $\theta$ of that ball, and so to finish the proof of the theorem it is enough to prove that the factor is positive at $\theta=\Theta$ for some suitable choice of $t$ and $\lambda$; then $\delta$ can be chosen small enough to ensure that the factor in parentheses remains positive when $\theta \in (\Theta-\delta,\Theta]$. When $\theta=\Theta$ we have $\mu_1(B_\Theta)\sin^2(\Theta)=1$ and so the task simplifies to finding $t$ and $\lambda$ such that
\begin{equation}\label{eq_derivativeTheta}
\mu_1(B_\Theta)g(t)^2-\frac{(M_{\om_\lambda})_{1,1}}{\pi}\frac{g(t)^2}{\sin^2 t }-\frac{(M_{\om_\lambda})_{2,2}}{\pi}g'(t)^2 > 0 ,
\end{equation}
where $g$ is the radial part of the eigenfunction for $B_\Theta$, with $\theta=\Theta$. 

Let $\lambda=1+\eta$ for some small $\eta>0$ to be fixed later. Using   formula \eqref{eq_computation_M} as previously, we have
\[\frac{(M_{\om_\lambda})_{1,1}}{\pi}=\frac{2}{1+\lambda^{-2}}= 1+\eta+\mathcal{O}(\eta^2) 
\]
and similarly
\[\frac{(M_{\om_\lambda})_{2,2}}{\pi}=\frac{2}{1-\lambda^{-2}}= \frac{1}{\eta}+\mathcal{O}(1) .
\]

Note also that $g'(\Theta)=0$ by the Neumann boundary condition (recall we have chosen $\theta=\Theta$), while the equation \eqref{gequation} satisfied by $g$ implies $g''(\Theta)=0$ (a special property at the radius $\Theta$, where $\mu_1(B_\Theta)\sin^2(\Theta)=1$) and $g^{\prime\prime\prime}(\Theta)=-2(\cos \Theta /\sin^3 \Theta) g(\Theta)>0$. Note $\cos \Theta < 0$ since $\Theta$ is larger than $\pi/2$. Consequently, writing $t=\Theta-\tau$ where $\tau>0$ is to be determined, we find 
\begin{align*}
g(\Theta-\tau) &=g(\Theta)+\frac{1}{3}\frac{\cos \Theta}{\sin^3 \Theta}g(\Theta)\tau^3+\mathcal{O}(\tau^4) , \\
g'(\Theta-\tau) &=-\frac{\cos \Theta}{\sin^3 \Theta}g(\Theta)\tau^2+\mathcal{O}(\tau^3) , \\
\frac{g(\Theta-\tau)^2}{\sin^2(\Theta-\tau)} & =
\frac{g(\Theta)^2}{\sin^2 \Theta} \left( 1 + 2 \tau \cot \Theta +\mathcal{O}(\tau^2) \right) .
\end{align*}
Hence the left side of \eqref{eq_derivativeTheta} equals
\begin{align*}
\mu_1(B_\Theta) & g(\Theta-\tau)^2-\frac{(M_{\om_{1+\eta}})_{1,1}}{\pi}\frac{g(\Theta-\tau)^2}{\sin^2(\Theta-\tau)} -\frac{(M_{\om_{1+\eta}})_{2,2}}{\pi}g'(\Theta-\tau)^2\\
&=\mu_1(B_\Theta)(g(\Theta)^2+\mathcal{O}(\tau^3))-(1+\eta+\mathcal{O}(\eta^2))\frac{g(\Theta)^2}{\sin^2 \Theta} \left( 1 + 2 \tau \cot \Theta +\mathcal{O}(\tau^2) \right) \\
&\qquad \qquad \qquad -(\eta^{-1}+\mathcal{O}(1))\left(-\frac{\cos \Theta}{\sin^3 \Theta}g(\Theta)\tau^2+\mathcal{O}(\tau^3)\right)^{\!\!2}\\
&=\left( -2\tau\cot \Theta-\eta \right) \frac{g(\Theta)^2}{\sin^2 \Theta} +\mathcal{O}\left(\tau^2+\eta^2+\frac{\tau^4}{\eta} \right) 
\end{align*}
by substituting $\mu_1(B_\Theta)=1/\sin^2 \Theta$ and simplifying. We now specify $\tau$ in terms of $\eta$ by requiring $\eta=-\tau \cot \Theta>0$.  Then \eqref{eq_derivativeTheta} holds for any small enough choice of $\eta$, because its left side equals 
\begin{align*}
\eta \, \frac{g(\Theta)^2}{\sin^2 \Theta} + \mathcal{O}(\eta^2) > 0 .
\end{align*}
This inequality finishes the proof of Proposition \ref{pr:nearTheta}. 
\end{proof}

\section{Numerical and heuristic analysis of a new class of domains}\label{sec_num}

 In order to understand how low the value of $\Theta-\delta$ in Theorem \ref{main_th} could be, in this section we are interested in a different kind of geometric perturbation of a spherical cap. This kind of counterexample is loosely inspired by the numerical optima found by Martinet \cite{M23}. 
The arguments we bring in this section are not rigorous enough to produce a fully mathematically justified lower value for  $\Theta-\delta$. However, our analysis is supported by numerical computations, gives a good hint of a new class of competitive geometries, and could possibly be made rigorous with extra effort. The sets we deal with consist  of the union of a spherical cap (geodesic ball) and a graph-like domain. 
For a geodesic cap $B_\theta$ having its center at $e_3$, define the domain 
\[
\Om_\e(\theta) = B_{\theta^\e} \cup \{ |x_1| < \e \} \cup \{ |x_2| < \e \}
\]
to be the union of a cap and two strips, where the aperture $\theta^\e < \theta$ is chosen such that $|\Om_\e(\theta)|=|B_\theta|$. Such domains, for several values of $\theta$ and $\e$, are depicted Figure ~\ref{fig:sphere_strips}. We call them ``helmet'' domains, although if they were turned upside down they could also be called ``basket'' domains.

\begin{figure}[!htp]
    \centering
    \includegraphics[width=0.26\textwidth]{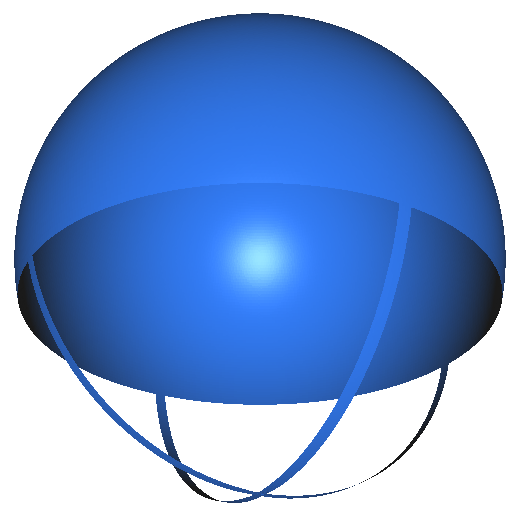}
    \includegraphics[width=0.26\textwidth]{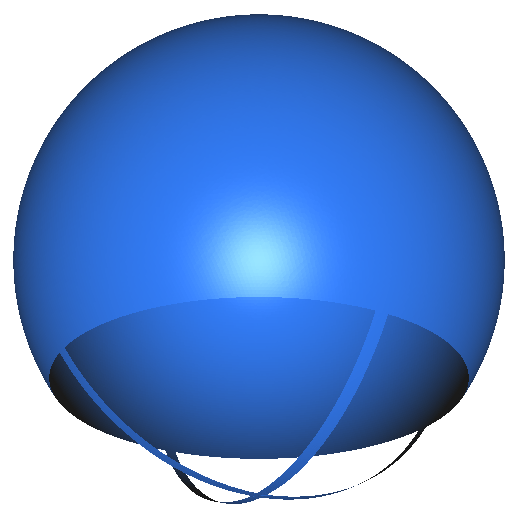}
    \includegraphics[width=0.26\textwidth]{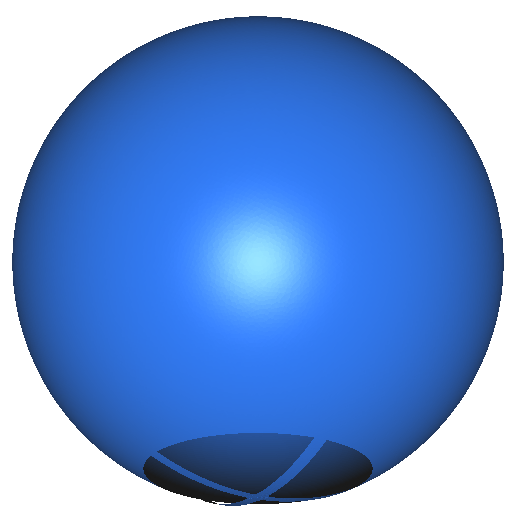}
    \includegraphics[width=0.26\textwidth]{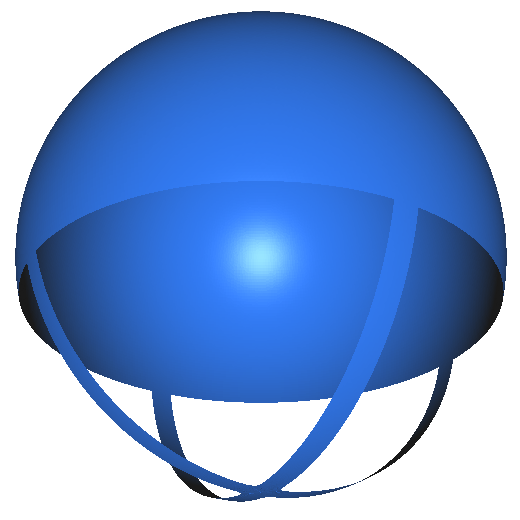}
    \includegraphics[width=0.26\textwidth]{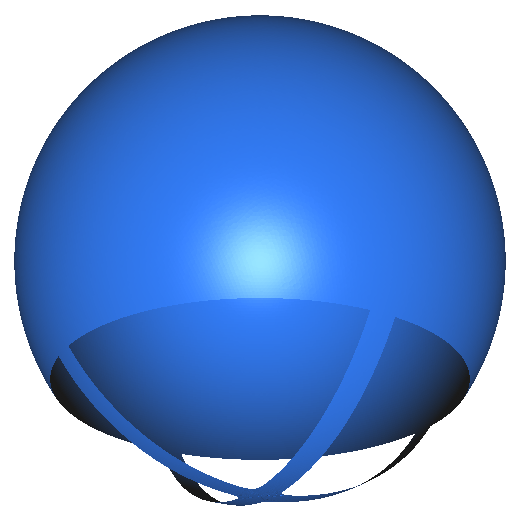}
    \includegraphics[width=0.26\textwidth]{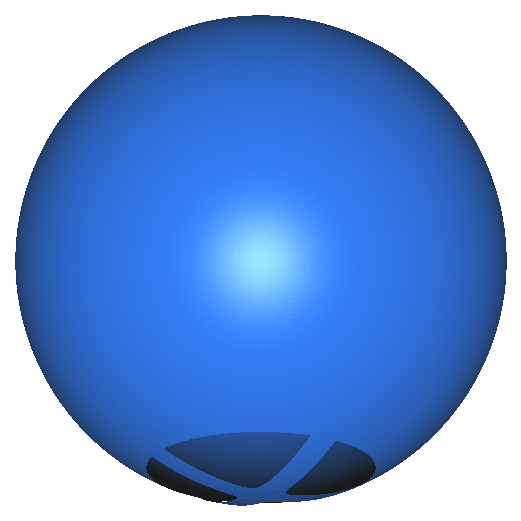}
    \caption{Sets $\Omega_\e(\theta)$ for several values of the aperture $\theta$ and for strip thickness $\e=0.05$ (top) and $\e=0.1$ (bottom). Here $e_3$ is oriented vertically.}
    \label{fig:sphere_strips}
\end{figure}

Using the finite element software FreeFem++ \cite{FF}, we can generate meshes on such domains and compute their eigenvalues. Figure ~\ref{fig:comparison_eigenvalues} plots the value of $\mu_1(\Omega_\e(\theta))$ as a function of $\theta$, for different values of $\e>0$, each of which provides counterexamples to maximality of the cap ($\e=0$) over slightly different intervals of $\theta$. 
\begin{figure}[!htp]
    \centering
    \includegraphics[width=\textwidth]{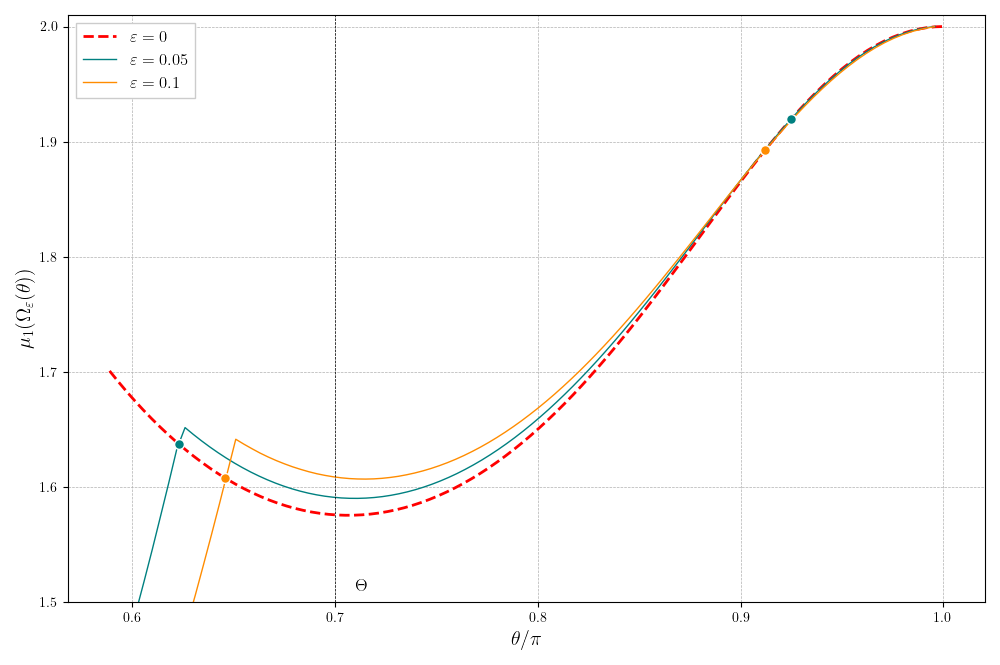}
    \caption{Value of $\mu_1(\Omega_\e(\theta))$ as a function of $\theta/\pi$, for $\e=0.05$ and $\e=0.1$. The red curve $\e=0$ corresponds to the spherical cap $B_\theta$. These graphs provide numerical counterexamples to maximality of $\mu_1$ at the cap, for cap apertures $\theta \in (0.623\pi, 0.925\pi)$ for $\e=0.05$ and $\theta \in (0.646\pi, 0.912\pi)$ for $\e=0.1$. The vertical dotted line is at $\Theta \simeq 0.7\pi$; recall that Theorem \ref{main_th} provides rigorous counterexamples to the right of that aperture.}
    \label{fig:comparison_eigenvalues}
\end{figure}

This figure deserves several comments. Most importantly, the helmet counterexamples work for certain aperture values $\theta$ that are substantially smaller than $\Theta$, which Theorem \ref{main_th} cannot do due to the very nature of the earlier construction. On the other hand, for a fixed $\e$ the helmet domains fail to provide counterexamples when $\theta$ is close to $\pi$, whereas Theorem \ref{main_th} holds all the way up to $\pi$: in other words, as $\e$ gets small we obtain two new critical angles $\Theta^\prime,\Theta^{\prime\prime}$ with
\[\frac{\pi}{2}<\Theta^\prime<\Theta<\Theta^{\prime\prime} \leq \pi\]
such that a helmet domain gives a counterexample for $\theta\in (\Theta^\prime,\Theta^{\prime\prime})$.\\

Next, the two eigenvalue curves in the figure exhibit downward-pointing corners around the aperture $\Theta^\prime (< \Theta)$, leading the current constructions to stop being counterexamples as the aperture gets smaller. It is shown in Figure \ref{fig:eigenfunctions} that this corner is linked to the change of nature of the eigenfunction associated to $\mu_1$ (with the eigenfunction transitioning to one that is concentrated on the strips) and is a crossing point for which we have $\mu_1=\mu_2=\mu_3$. Lastly, it seems that decreasing the width $\e$ of the strip allows one to get counterexamples for smaller values of $\theta$. A question of interest would be to provide the precise asymptotic behaviour of $\mu_1(\Omega_\e(\theta))$ when $\e$ goes to $0$, with the purpose to determine the smallest $\theta$ that can yield a counterexample of this type.

\begin{figure}[!htp]
  \centering
  \includegraphics[width=0.3\textwidth]{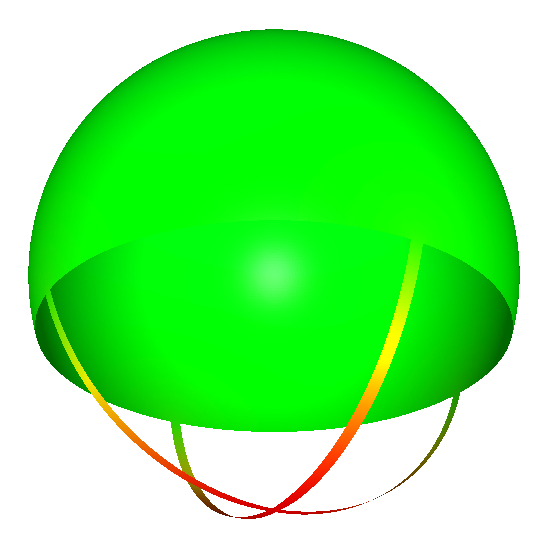}
  \includegraphics[width=0.3\textwidth]{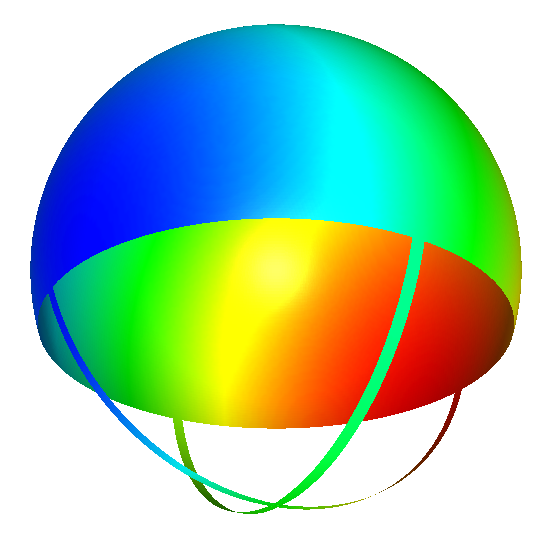}
  \caption{Eigenfunctions associated to $\mu_1$ for $\e=0.05$, for $\theta = 1.96$ (left) and $\theta=1.97$ (right). The red color corresponds to the value $1$ while deep blue corresponds to $-1$.}
  \label{fig:eigenfunctions}
\end{figure}

For replicability purposes, the examples used to generate Figure \ref{fig:comparison_eigenvalues} are available at \url{https://github.com/EloiMartinet/PuncturedSpheres}.\bigbreak

Let us now give a heuristic (and non-rigorous) justification of these counterexamples,   based on the results of Arrieta \cite{Ar95} proved in the case of simple eigenvalues in ``dumbbell'' domains. To reduce the number of leaps of faith, let us proceed slightly differently to above by no longer attempting to preserve the original area: simply add the straps to the cap $B_\theta$ by letting
\[
H_{\theta,\eps}=B_\theta\cup\{x:|x_1|\leq \eps\text{ or }|x_2|\leq \eps\}
\]
where $H$ stands for ``Helmet'', and compare with the cap 
\[B_{\theta(\eps)} \quad \text{ where }\theta(\eps)\text{ is chosen to satisfy }|B_{\theta(\eps)}|=|H_{\theta,\eps}|.\]

Let us first compute an asymptotic for $\mu_1(B_{\theta(\eps)})$, with a shape derivative. Since $|H_{\theta,\eps}|=|B_\theta|+8(\pi-\theta)\eps+o(\eps)$, we have $\theta(\eps)=\theta+ \frac{8(\pi-\theta)\eps}{|\partial B_\theta|}+o(\eps)$. As a consequence,
\begin{align*}
\mu_1(B_{\theta(\eps)})&=\mu_1(B_\theta)+\frac{8(\pi-\theta)\eps}{|\partial B_\theta|}\int_{\partial B_\theta}\left(|\nabla u_1|^2-\mu_1(B_\theta)u_1^2\right)\\
&=\mu_1(B_\theta)+4(\pi-\theta)\eps
\left(\frac{1}{\sin^2 \theta}-\mu_1(B_\theta)\right)g(\theta)^2\\
\end{align*}
where $u_1$ is an eigenfunction associated to $\mu_1(B_\theta)$, of the form $g(t) \sin \phi$ or $g(t) \cos \phi$. \bigbreak

Let us now compute an asymptotic for $\mu_1(H_{\theta,\eps})$: following Arrieta \cite[Theorem 2.5]{Ar95} we expect the Neumann spectrum of $H_{\theta,\eps}$ to be given by the union (counting multiplicities) of the spectra $(\nu_i^\e)_{i\geq 1}$ and $(\mu_i^\e)_{i\geq 0}$, where

\begin{itemize}
\item $\nu_i^\e$ corresponds to an eigenfunction for which the $L^2$ norm is concentrated on the ``helmet straps'' $H_{\theta,\eps} \setminus B_{\theta}$. One expects that $\nu_i^\e=\nu_i+o(1)$ as $\e\to 0$, where the $(\nu_i)_{i\geq 1}$ are the solutions to the 1-dimensional Dirichlet eigenvalue problem on the graph 
\[
G:=(\{x_1=0\}\cup \{x_2=0\})\setminus B_\theta
\]
on the sphere, with boundary $\partial G:=(\{x_1=0\}\cup \{x_2=0\})\cap \partial B_\theta$, and where the eigenvalue problem is 
\[\begin{cases}
-\Delta_G w_i=\nu_i w_i&\text{ in }G, \\
w_i=0 & \text{ on }\partial G ,
\end{cases} \]
and the first equation is understood as $-\partial_\tau^2 w_i=\nu_i w_i$ along the edges of the graph, with a Kirchhoff condition at the (only) vertex. Explicitly, one can show the first eigenvalue is 
\[
\nu_1= \frac{\pi^2}{4(\pi-\theta)^2}
\]
associated to the eigenfunction 
\[
w_1\left(\pi+\tau,\frac{k\pi}{2}\right)=\cos\left(\frac{\pi \tau}{2(\pi-\theta)}\right) , \qquad \tau\in [-(\pi-\theta),\pi-\theta] , \quad k=0,1 .
\]

In particular $\nu_1>\mu_1(B_\theta)$ if and only if $(\pi-\theta)\sqrt{\mu_1(B_\theta)}<\frac{\pi}{2}$. Letting $\Theta^\prime$ be the value for which $(\pi-\theta)\sqrt{\mu_1(B_\theta)}=\frac{\pi}{2}$, it may be checked that $\Theta^\prime \simeq 0.61\pi$ and that the condition $(\pi-\theta)\sqrt{\mu_1(B_\theta)}<\frac{\pi}{2}$ holds for all $\theta>\Theta^\prime$. Accordingly, from now on we suppose $\theta \in (\Theta^\prime,\pi)$.

\smallskip
\item   $\mu_i^\e$ corresponds to a perturbation of $\mu_i(B_\theta)$, and is associated to an eigenfunction $u_i^\e$ of $H_{\theta,\eps}$ that is an $L^2$-perturbation of an eigenfunction associated to $\mu_i(B_\theta)$, denoted $u_i$. In this case an asymptotic of $\mu_i^\eps$ is 
\[\mu_i^\eps=\mu_i(B_\theta)+\eps\int_{G}\left(|\nabla_G v_i|^2-\mu_i(B_\theta)|v_i|^2\right)+o(\eps),\]
where $v_i$ is defined as the solution to
\[
\begin{cases}
-\Delta_G v_i=\mu_i(B_\theta) v_i&\text{ in }G ,\\
v_i=u_i & \text{ on }\partial G ,
\end{cases} 
\]
with again a Kirchhoff condition at the vertex. This $v_i$ exists and is unique only when $\mu_i(\theta)\neq \nu_j$ for all $j$, which is true for $i=1$ when we suppose $\theta>\Theta^\prime$. 

\end{itemize}

Under the hypothesis that $\theta>\Theta'$, we know that for a small enough $\eps$ the first eigenvalue of $H_{\theta,\eps}$ is of the form $\mu_1^\eps$ as described above. For $i=1$, the Neumann eigenvalue $\mu_1(B_\theta)$ is associated to the eigenfunctions
\[
u_1(t,\phi)=g(t)\cos \phi ,\qquad u_2(t,\phi)=g(t)\sin \phi,
\]
where $t$ is the angle or geodesic distance from the north pole and $\phi$ is the angle around the line of latitude. We may compute directly that the first two eigenfunctions on the graph concentrate on the two ``straps'', respectively: 
\[
v_1\left(\pi+\tau,0\right)=-\frac{\sin\left(\sqrt{\mu_1(B_\theta)} \, \tau\right)}{\sin\left(\sqrt{\mu_1(B_\theta)}(\pi-\theta)\right)} , \qquad  v_1\left(\pi+\tau,\pm \frac{\pi}{2}\right)=0 ,
\]
for $\tau\in [-(\pi-\theta),\pi-\theta]$, and similarly
\[
v_2\left(\pi+\tau,0\right)=0, \qquad v_2\left(\pi+\tau,\pm\frac{\pi}{2}\right)=-\frac{\sin\left(\sqrt{\mu_1(B_\theta)} \, \tau\right)}{\sin\left(\sqrt{\mu_1(B_\theta)}(\pi-\theta)\right)}g(\theta) .
\]

So for $i=1,2$ we have
\begin{align*}
& \int_{G}\left(|\nabla_G v_i|^2-\mu_i(B_\theta)v_i^2\right)d\mathcal{H}^1 \\
&=\frac{g(\theta)^2\mu_1(B_\theta)}{\sin^2\left(\sqrt{\mu_1(B_\theta)}(\pi-\theta)\right)}\int_{-(\pi-\theta)}^{\pi-\theta}\left(\cos^2\left(\sqrt{\mu_1(B_\theta)} \, \tau\right)-\sin^2\left(\sqrt{\mu_1(B_\theta)} \, \tau\right)\right)d\tau\\
&=2g(\theta)^2\sqrt{\mu_1(B_\theta)}\cot\left(\sqrt{\mu_1(B_\theta)}(\pi-\theta)\right) .
\end{align*}
Hence
\[
\mu_1(H_{\theta,\eps})=\mu_1(B_\theta)+4\e g(\theta)^2\sqrt{\mu_1(B_\theta)}\cot\left(\sqrt{\mu_1(B_\theta)}(\pi-\theta)\right)+o(\e).
\]
As a consequence, for any $\theta>\Theta'$ we obtain $\mu_1(H_{\theta,\eps})>\mu_1(B_{\theta(\eps)})$ for a small enough $\e$ as soon as the inequality
\[\sqrt{\mu_1(B_\theta)}\cot\left(\sqrt{\mu_1(B_\theta)}(\pi-\theta)\right)>(\pi-\theta)\left(\frac{1}{\sin^2 \theta}-\mu_1(B_\theta)\right)\]
is verified. When $\theta\in (\Theta^\prime,\Theta]$ the left-hand side is positive and the right-hand side is nonpositive, so that the inequality holds on $(\Theta^\prime,\Theta^{\prime\prime})$, for some $\Theta^{\prime\prime}$ slightly larger than $\Theta$. Numerics suggest that the inequality holds for all $\theta\in (\Theta',\pi)$, so that helmet counterexamples can be constructed for all such $\theta$, but we do not prove this claim.

\section{No asymptotic counterexample with a single hole}
\label{sec_asymptotic}

Here we are interested in the possibility of finding a simply connected counterexample. Precisely, consider a set $\Om^\eps$ of the form

\[
\Om^\eps=\Sd\setminus \psi(\eps\om) ,
\]
where $\om\subset\R^2$ is a smooth simply connected set with finite measure, where $\eps > 0$ is small and $\psi:\R^2\to \Sd$ is the stereographic projection  sending the origin $0$ to the south pole $-e_3$, with Jacobian $1$ at $0$.

Is it possible to obtain, for a fixed $\om$ and arbitrarily small $\eps$, a counterexample to the maximality of the spherical cap among domains of the same area? We prove that such a counterexample  works neither for the functional $\mu_1$ nor  for $\mu_1+\mu_2$ or $\left( \frac{1}{\mu_1}+\frac{1}{\mu_2} \right)^{\! -1}$.  Our  result does not imply that the spherical cap maximizes $\mu_1$ among ``large'' spherical domains, since we do not explore general sequences of domains expanding to fill the sphere,  merely domains  that are given  by the complement of a dilation of a fixed simply connected set $\omega$.

Since $\mu_1=\mu_2=\mu_3=2$ on the full sphere $\Sd$, with eigenfunctions given by the $L^2$-normalized coordinate functions $u_i=\sqrt{3/4\pi} \, x_i$ for $i=1,2,3$, Proposition \ref{Prop_limholes} with $M=\Om=\Sd$ implies that the eigenvalues of the sphere with hole $\psi(\eps\om)$  are of the form
\[\mu_i(\Om^\eps)=2+\eps^2\kappa_i+o(\eps^2)\]
where $\kappa_i$ is the $i$-th eigenvalue of the matrix
\[
\big( \, 2|\om|v_i(0)v_j(0)-\langle M_\om \nabla v_i(0),\nabla v_i(0)\rangle \, \big)_{1\leq i,j\leq 3}
\]
and where we wrote $v_i=u_i\circ \psi$. In this case we have $v_1(0)=v_2(0)=0, v_3(0)=-\sqrt{3/4\pi}$, and $\nabla v_3(0)=0, \nabla v_i(0)=\sqrt{\frac{3}{4\pi}}\, e_i$ for $i=1,2$. Thus this matrix has the form
\[
\frac{3}{4\pi}\begin{pmatrix}
-(M_\om)_{1,1} & -(M_\om)_{1,2} & 0 \\
-(M_\om)_{2,1} & -(M_\om)_{2,2} & 0 \\
0 & 0 & 2|\om|
\end{pmatrix} .
\]
So, we obtain
\begin{align*}
\mu_1(\Om^\eps)+\mu_2(\Om^\eps)&=4-\frac{3\text{Tr}(M_\om)}{4\pi}\eps^2+o(\eps^2) , \\
\mu_3(\Om^\eps)&=2+\frac{3|\om|}{2\pi}\eps^2+o(\eps^2) . 
\end{align*}
We recall that using the formula \eqref{eq_computation_M}, we have $\text{Tr}(M_\om)=4\pi \cp[\om]^2$. As a consequence

\[\mu_1(\Om^\eps)+\mu_2(\Om^\eps)=4-3\cp[\om]^2\eps^2+o(\eps^2).\]

As observed above, the disk has the smallest capacity among sets of given measure. Hence in the limit as $\eps \to 0$, the sum $\mu_1(\Om^\eps)+\mu_2(\Om^\eps)$ of the first two eigenvalues is smaller than the corresponding sum when $\om$ is a disk, that is, when $\Om^\eps$ is a spherical cap. The same conclusion follows for the first eigenvalue $\mu_1(\Om^\eps)$, since $\mu_1=\mu_2$ for the spherical cap. Analogous computations may be done for $\left( \frac{1}{\mu_1}+\frac{1}{\mu_2}\right)^{\! -1}$.

\section{Proof of formulas \eqref{W_n_M_n}, \eqref{formula_shapederivative}, \eqref{eq_Fni}, and Lemma \ref{lem_estansatz}, and harmonic extension to the holes}
\label{sec:lem_estansatz}

Here we prove some results stated in Section \ref{sec_conv}. Consider the setting and notations of Proposition \ref{Prop_limholes}.

\subsection*{Proof of formula \eqref{W_n_M_n}}
The harmonic vector field $W_n$ vanishes at infinity and, after transforming infinity to the origin by the Kelvin inversion, we know $W_n$ is locally linear. That is,  
\begin{equation} \label{WC}
W_n(x) = \frac{C_n x}{2\pi |x|^2} + \mathcal{O}_{|x|\to\infty} \! \left(\frac{1}{|x|^2}\right) 
\end{equation}
for some $2 \times 2$ matrix $C_n$. We will show for formula \eqref{W_n_M_n} that $C_n=M_n$.

From now on, we drop the subscript $n$ on the vector field $W$, the matrices $C$ and $M$, and the set $\om$. The goal is to relate the matrix entries by $M_{p,q}=C_{p,q}$. Let $R \gg 1$. Writing $W_p$ for the $p$-th component of $W$, we find 
\begin{align*}
& \delta_{p,q}  |\om| + \int_{D_R \setminus \om} \nabla W_p \cdot \nabla W_q \\
& = \int_\om \nabla \cdot (x_q e_p) + \int_{D_R \setminus \om} \nabla \cdot (W_p \nabla W_q ) \\
& = \int_{\partial \om} \left( x_q \nu_p - W_p \, \partial_\nu W_q \right) + \int_{\partial D_R} W_p \, \partial_\nu W_q \qquad \text{by the divergence theorem} \\
& = \int_{\partial \om} \left( - x_q \, \partial_\nu W_p + W_p \nu_q \right) + \mathcal{O}_{|x|\to\infty} \! \left(\frac{1}{R^2}\right) 
\end{align*}
by the boundary condition $\partial_\nu W_n=-\nu$ on $\partial \om$. Applying Green's theorem and recalling the convention that the normal vector points outward from both $\om$ and the disk $D_R$, we see 
\begin{align*}
& \delta_{p,q} |\om| + \int_{D_R \setminus \om} \nabla W_p \cdot \nabla W_q \\
& = \int_{\partial D_R} \left( - x_q \, \partial_\nu W_p + W_p \nu_q \right) + \int_{D_R \setminus \om} \left( x_q \Delta W_p - W_p \Delta x_q \right) + \mathcal{O}_{|x|\to\infty} \! \left(\frac{1}{R^2}\right) .
\end{align*}
The second integral on the right vanishes since both $W_p$ and $x_q$ are harmonic. For the first integral on the right, observe that the normal vector $\nu = x/|x|$ is radially outward on the circle, and use \eqref{WC} to compute that 
\[
W_p = \frac{C_p \cdot x}{2\pi |x|^2} + \mathcal{O}_{|x|\to\infty} \! \left(\frac{1}{|x|^2}\right) , \qquad - x_q \, \partial_\nu W_p = x_q \frac{C_p \cdot x}{2\pi |x|^3} + \mathcal{O}_{|x|\to\infty} \! \left(\frac{1}{|x|^2}\right) ,
\]
where $C_p$ is the $p$-th row of $C$. It is now straightforward to integrate over the circle $\partial D_R$, deducing that 
\[
\delta_{p,q}  |\om| + \int_{D_R \setminus \om} \nabla W_p \cdot \nabla W_q 
= C_{p,q} + \mathcal{O}_{|x|\to\infty} \! \left(\frac{1}{R}\right) .
\]
Letting $R \to \infty$ shows that $M_{p,q}=C_{p,q}$.

\subsection*{Proof of formula \eqref{formula_shapederivative}}
The goal is to show
\[
-\int_{\Om}(Ku_i)u_j+\int_{\partial\Om}\langle \eta , \nabla u_i\rangle \, u_j=\int_{\partial\Om}\left(\langle \nabla u_i,\nabla u_j \rangle -\mu u_i u_j\right)\langle \phi,\nu\rangle
\]
where $K$ and $\eta$ are defined at the beginning of Section \ref{subsec_ansatz}. Similar computations may be found in \cite[Th.\ 5.7.2.]{HenrotPierre}. In our case we rely on the computations in Fall and Weth \cite[Sections 2 and 3]{Fall}. 

Since the metric varies in this section, we make its dependence appear explicitly: $\langle \cdot,\cdot\rangle_{t}$ is the scalar product associated to $g_t$ and $\nabla_{t}$ is the gradient, so that $\nabla_t \, \cdot$ is the divergence. Also, $\Delta_t=\nabla_t\cdot\nabla_t$ is the Laplacian, $dm_t$ is the area measure associated to $g_t$, and $\nu_t$ is the outward normal unit vector on $\partial \Om$, and $ds_t$ is the boundary measure on $\partial\Om$. We also denote by $D$ the Levi-Civita connection associated to $g$. 
Green's theorem says
\begin{align*}
-\int_{\Om}u_j\Delta_{t}u_i \, dm_t+\int_{\partial\Om} u_j \, du_i(\nu_t) \, ds_{t}=\int_{\Om}\langle \nabla_{t}u_i,\nabla_{t}u_j\rangle_{t} \, dm_t .
\end{align*}

A commutator identity will be useful. For vector fields $\phi$ and $\psi$ and functions $h$, one has
\[
\langle [\psi,\phi],\nabla h \rangle = \psi \, \langle \phi,\nabla h \rangle - \phi \, \langle \psi, \nabla h \rangle .
\]
When $\psi$ is chosen to be a gradient field, say $\psi = \nabla f$, the identity implies that 
\begin{equation} \label{commutator}
\langle [\nabla f,\phi],\nabla h \rangle = \langle \nabla f , \nabla \langle \phi, \nabla h \rangle \rangle - \phi \, \langle \nabla f , \nabla h \rangle  .
\end{equation}
We have 
\begin{align*}\allowdisplaybreaks
& 
\left.\frac{d}{dt}\right|_{t=0} \int_{\Om}\langle \nabla_{t}u_i,\nabla_{t}u_j\rangle_{t} \, dm_t \\
& 
=\int_{\Om}\left(\langle \nabla u_i,\nabla u_j\rangle \, \nabla\cdot\phi- \left\{ \langle  D_{\nabla u_i}\phi,\nabla u_j\rangle + \langle  D_{\nabla u_j}\phi,\nabla u_i\rangle\right\}\right)dm\\
& 
\hspace{2cm} \text{by using Fall and Weth \cite[formulas (7) and (14)]{Fall}}\\
& 
=\int_{\partial\Om}\langle \nabla u_i,\nabla u_j\rangle\ \langle\phi,\nu\rangle \, ds-\int_{\Om} \langle \phi,\nabla\langle \nabla u_i,\nabla u_j\rangle\rangle \, dm \qquad \text{by parts} \\
& \qquad -\int_{\Om}\left\{ \langle  D_\phi \nabla u_i ,\nabla u_j\rangle+\langle  D_\phi \nabla u_j ,\nabla u_i\rangle + \langle [\nabla u_i,\phi],\nabla u_j \rangle + \langle [\nabla u_j,\phi],\nabla u_i \rangle \right\} dm\\
& 
 \hspace{3cm} \text{by the torsion-free property of the connection}\\
& 
=\int_{\partial\Om}\langle \nabla u_i,\nabla u_j\rangle\ \langle\phi,\nu\rangle \, ds-\int_{\Om} \phi \, \langle \nabla u_i,\nabla u_j\rangle \, dm \\
& \qquad -\int_{\Om} \left\{ \phi \, \langle  \nabla u_i ,\nabla u_j\rangle + \langle \nabla u_i , \nabla \langle \phi,\nabla u_j\rangle \rangle + \langle \nabla u_j , \nabla \langle \phi,\nabla u_i\rangle \rangle - 2\phi \, \langle \nabla u_i , \nabla u_j \rangle \right\} dm
\end{align*}
by the Leibniz property of the connection, and the commutator identity \eqref{commutator}. Cancelling terms, we find 
\begin{align*}
& 
\left.\frac{d}{dt}\right|_{t=0} \int_{\Om}\langle \nabla_{t}u_i,\nabla_{t}u_j\rangle_{t} \, dm_t \\
&
=\int_{\partial\Om}\langle \nabla u_i,\nabla u_j\rangle\ \langle\phi,\nu\rangle \, ds -\int_{\Om} \left\{ \langle \nabla u_i , \nabla \langle \phi,\nabla u_j\rangle \rangle + \langle \nabla u_j , \nabla \langle \phi,\nabla u_i\rangle \rangle \right\} dm \\
& 
=\int_{\partial\Om}\langle \nabla u_i,\nabla u_j\rangle\ \langle\phi,\nu\rangle \, ds+\int_{\Om}\left\{\langle \phi,\nabla u_j\rangle \Delta u_i + \langle \phi,\nabla u_i\rangle \Delta u_j \right\}dm \\
& 
\hspace{2cm} \text{by Green's theorem since $\langle\nabla u_i,\nu\rangle=0$ and $\langle\nabla u_j,\nu\rangle=0$ on $\partial \Om$} \\
& =\int_{\partial\Om}\langle \nabla u_i,\nabla u_j\rangle\ \langle\phi,\nu\rangle \, ds-\mu\int_{\Om}\langle \phi,\nabla (u_iu_j) \rangle \, dm
\end{align*}
since $\Delta u_i=-\mu u_i$ and $\Delta u_j=-\mu u_j$.

Hence from the divergence theorem we conclude 
\begin{equation} \label{FallWeth}
\left.\frac{d}{dt}\right|_{t=0} \int_{\Om}\langle \nabla_{t}u_i,\nabla_{t}u_j\rangle_{t} \, dm_t 
=\int_{\partial\Om}\left\{\langle \nabla u_i,\nabla u_j\rangle-\mu u_i u_j\right\} \langle\phi,\nu\rangle \, ds+ \mu \int_{\Om}u_iu_j(\nabla\cdot\phi) \, dm .
\end{equation}

Since $Kf$ and $\eta$ are respectively the derivatives at $t=0$ of $\Delta_t f$ and $\nu_t$, we have 
\begin{align*}
& \hspace*{-0.75cm} -\int_{\Om} u_jK u_i \, dm+\int_{\partial\Om}u_j du_i(\eta) \, ds \\
& =\left.\frac{d}{dt}\right|_{t=0}\left[-\int_{\Om}u_j\Delta_t u_i \, dm_t+\int_{\partial\Om}u_j du_i(\nu_t) \, ds_t+\int_{\Om}u_j\Delta u_i \, dm_t-\int_{\partial\Om}u_j du_i(\nu) \, ds_t\right]\\
&
=\left.\frac{d}{dt}\right|_{t=0}\left[\int_{\Om} \langle \nabla_t u_i,\nabla_t u_j\rangle_t \, dm_t +\int_{\Om}u_j\Delta u_i \, dm_t \right] \qquad \text{since $du_i(\nu)=0$} \\
& 
=\int_{\partial\Om}\left\{\langle \nabla u_i,\nabla u_j\rangle-\mu u_i u_j\right\} \langle\phi,\nu\rangle \, ds
\end{align*}
by \eqref{FallWeth}, where in the last line we used again that $\Delta u_i=-\mu u_i$, and that $\left. \frac{d\ }{dt} \right|_{t=0} dm_t = (\nabla \cdot \phi) \, dm$ (see \cite[formula (7)]{Fall}).

\subsection*{Proof of formula \eqref{eq_Fni}}
Recall $F_{n,i}$ was defined in \eqref{def_Fni}. Write $\tilde{u}_j = u_j \circ \psi_n$ in the proof below. Since $F_{n,i}$ is supported in $B_r$, the left side of \eqref{eq_Fni} equals 
\begin{align*}
-\int_{\psi_n(B_r)} (F_{n,i}\circ\psi_n^{-1})u_j
& = - \int_{B_r} \rho_n F_{n,i} \tilde{u}_j \\
& = - \int_{B_r \setminus \{0\}} \tilde{u}_j (\Delta + \mu \rho_n) g && \text{by \eqref{def_Fni}, with $g$ as defined below,} \\
& = - \int_{B_r \setminus \{0\}} (\tilde{u}_j \Delta g - g \Delta \tilde{u}_j ) && \text{since $\Delta \tilde{u}_j = -\mu \rho_n \tilde{u}_j$,} \\
& = \lim_{s \to 0} \int_{\partial B_s} (\tilde{u}_j \partial_r g - g \partial_r \tilde{u}_j) && \text{by Green's theorem,}
\end{align*}
where the compactly supported function $g$ on $B_r$ is
\[
g(x) = \zeta(x)\left(\frac{\langle x,M_n\nabla \tilde{u}_i(0)\rangle}{2\pi|x|^2}+\frac{\rho_n(0)\mu |\om_n|\tilde{u}_i(0)}{2\pi}\log|x|\right) .
\]

We consider each boundary integral in turn. First, the definition of $g$ implies that 
\begin{align*}
\int_{\partial B_s} g \partial_r \tilde{u}_j 
& = \int_{\partial B_s} \left( \frac{1}{2\pi s} \langle x/|x| , M_n\nabla \tilde{u}_i(0)\rangle \, \langle x/|x| , \nabla \tilde{u}_j(0)\rangle + o \! \left(\frac{1}{s}\right) \right) \\
& \to \frac{1}{2} \langle M_n\nabla \tilde{u}_i(0),\nabla \tilde{u}_j(0)\rangle
\end{align*}
as $s \to 0$, by parameterizing the circle with $x=(s\cos \phi,s\sin \phi)$. Next, 
\begin{align*}
\tilde{u}_j \partial_r g 
& = \left( \tilde{u}_j(0) + s \langle x/|x| , \nabla \tilde{u}_j(0) \rangle + o(s) \right) \\
& \hspace{2cm} \times \left( - \frac{1}{2\pi s^2} \langle x/|x| , M_n \nabla \tilde{u}_i(0) \rangle + \frac{1}{2\pi s} \rho_n(0) \mu |\om_n| \tilde{u}_i(0) \right) \\
& = - \frac{\tilde{u}_j(0)}{2\pi s^2} \langle x/|x| , M_n\nabla \tilde{u}_i(0) \rangle - \frac{1}{2\pi s} \langle x/|x| , \nabla \tilde{u}_j(0) \rangle \, \langle x/|x| , M_n \nabla \tilde{u}_i(0) \rangle \\
& \hspace{2cm} + \frac{1}{2\pi s} \rho_n(0) \mu |\om_n| \tilde{u}_i(0) \tilde{u}_j(0) + o \! \left(\frac{1}{s}\right) .
\end{align*}
Parameterizing the circle and integrating yields that 
\[
\int_{\partial B_s} \tilde{u}_j \partial_r g \to 0 - \frac{1}{2} \langle \nabla \tilde{u}_j(0),  M_n\nabla \tilde{u}_i(0)\rangle + \rho_n(0) \mu |\om_n| \tilde{u}_i(0) \tilde{u}_j(0) . 
\]
Combining the two boundary limits completes the proof of formula \eqref{eq_Fni}.

\subsection*{Proof of Lemma \ref{lem_estansatz}} To lighten the notation, we drop the index ``$i$'' in all that follows. Consider $z\in H^1(\Om^\eps)$ with unit norm. 

First we separate the estimate in Lemma \ref{lem_estansatz} into two parts, by Green's theorem: 
\begin{equation}\label{eq_decompositionregularboundary}
\int_{\Om^\eps}\left(\mu_{\eps}u_{\eps}z- \langle \nabla u_{\eps},\nabla z \rangle\right)=\underbrace{\int_{\Om^\eps}z(\Delta+\mu_{\eps})u_{\eps}-\int_{\partial\phi^{\eps^2}(\Om)}z\partial_\nu u_\eps}_{\text{regular term}}+\sum_{n=1}^{N}\underbrace{\int_{\partial(\eps\om_n)} (z\circ \psi_n) \, \partial_{\nu}(u_\eps\circ \psi_n)}_{\text{boundary term}} ,
\end{equation}
noting in the boundary term that $\psi_n$ is conformal and so the change of variable factor for the normal derivative compensates exactly for the tangential change of variable factor due to the arclength element. 

\subsubsection*{Boundary term}
Fix $n\in\{1,\hdots,N\}$ and write $\tilde{f}=f\circ\psi_n$ for $f=u,u_\eps,v,z$, so that the ``tilde'' version of each function is defined either on $B_r$ or  $B_r \setminus \eps \om_n$, as the case may be. The boundary term is $\int_{\partial (\eps \om_n)} \tilde{z} \, \partial_\nu \tilde{u}_\eps$. 

In terms of these functions, on $B_r \setminus \eps \om_n$ the ansatz \eqref{eq_ansatz} becomes 
\begin{align*}
\tilde{u}_{\eps}(x)&=\tilde{u}(x)+\eps^2 \tilde{v}(x)+\eps \zeta(x) \langle W_n(x/\eps),\nabla \tilde{u}(0)\rangle\\
& \hspace{1.35cm} +\eps^2 \zeta(x)  \left( -L_n(x/\eps):D^2 \tilde{u}(0)+\frac{\rho_n(0)\mu|\om_n|\tilde{u}(0)}{2\pi}\log \eps\right)\\
\mu_{\eps}&=\mu+\eps^2\kappa
\end{align*}
where $-\Delta \tilde{u}=\rho_n \mu \tilde{u}$ and $\tilde{v}$ verifies
\[
-\Delta \tilde{v} -\rho_n \mu \tilde{v}=\rho_n \kappa \tilde{u}+\rho_n F_n
\]
in $B_r$, where we recall that $F_n$ is defined in equation \eqref{def_Fni}. The normal derivative of $\tilde{u}_\eps$ can be computed from the ansatz, for $x\in\partial(\eps\om_n)$:
\begin{align*}
\partial_\nu \tilde{u}_\eps(x)&=\langle \nu(x),\nabla \tilde{u}(x)\rangle+ \langle (\partial_\nu W_n)(x/\eps),\nabla\tilde{u}(0)\rangle-\eps\partial_\nu L_n(x/\eps):D^2 \tilde{u}(0)+\eps^2\partial_\nu \tilde{v}(x)\\
&=\langle \nu(x),\nabla \tilde{u}(x)-\nabla \tilde{u}(0)- D^2\tilde{u}(0)x \rangle+\eps^2\partial_\nu \tilde{v}(x) 
\end{align*}
by using the boundary conditions $\partial_\nu W_n(x)=-\nu(x)$ and $\partial_\nu L_n(x)=x\otimes \nu(x)$ for $x \in \partial \om_n$, together with the fact that the unit normal vector remains invariant under scaling ($\nu_{\om_n}{\!}(x/\eps)=\nu_{\eps \om_n}(x)$ when $x \in \partial(\eps \om_n)$). 

Regard $\tilde{z}=z \circ \psi_n$ as being extended harmonically into $\eps\om_n$. The extended $\tilde{z}$ is still bounded (independent of $\eps$) in $H^1(B_r)$, by Lemma \ref{lem_extension_holes} below and using the scale-invariance of the Dirichlet energy. As a consequence, $\tilde{z}$ is also bounded in $L^p(B_r)$ for all $p<\infty$, and in particular is bounded in $L^{12}(B_{r})$.

From the definition of $v$ in \eqref{vdef}, we see $\Delta \tilde v\in L^{3/2}(B_{r})$ and so elliptic regularity implies $\tilde{v} \in W^{2,3/2}(B_{r/2})$ and so $\nabla \tilde v\in L^4(B_{r/2})$. (In fact, $\Delta \tilde{v} \in L^p$ for any $p\in [1,2)$ and so $\tilde{v} \in W^{2,p}(B_{r/2})$ and hence $\nabla v\in L^q(B_{r/2})$ for all $q < \infty$.) We decompose

\begin{align*}
\int_{\partial (\eps \om_n)} \tilde{z} \, \partial_\nu \tilde{u}_\eps
& = \int_{\partial(\eps\om_n)}\tilde{z} \, \langle \nabla \tilde{u}-\nabla\tilde{u}(0)-D^2\tilde{u}(0)x+\eps^2\nabla\tilde{v}, \nu\rangle \\
&=\int_{\eps\om_n}\nabla \tilde{z}\cdot \left(\nabla \tilde{u}-\nabla\tilde{u}(0)-D^2\tilde{u}(0)x+\eps^2\nabla\tilde{v}\right)\\
&+\int_{\eps\om_n}\tilde{z}\nabla\cdot\left(\nabla \tilde{u}-\nabla\tilde{u}(0)-D^2\tilde{u}(0)x+\eps^2\nabla\tilde{v}\right) 
\end{align*}
by the divergence theorem. In the last line, note that $\nabla \cdot D^2\tilde{u}(0)x=\Delta\tilde{u}(0)$. We bound terms separately as follows, using H\"{o}lder's inequality and that the area of $\eps \om_n$ equals $\eps^2 |\om_n|$:
\begin{align*}
\left|\int_{\eps\om_n}\nabla \tilde{z}\cdot \left(\nabla \tilde{u}-\nabla\tilde{u}(0)-D^2\tilde{u}(0)x\right)\right|&\leq C\Vert \tilde{z}\Vert_{H^1(B_r)}|\eps\om_n|^{\frac{1}{2}}\sup_{\eps\om_n}\left|\nabla \tilde{u}-\nabla\tilde{u}(0)-D^2\tilde{u}(0)x\right|\leq C'\eps^3 \\
\left|\int_{\eps\om_n}\tilde{z}\left(\Delta\tilde{u}-\Delta\tilde{u}(0)\right)\right|&\leq C\Vert\tilde{z}\Vert_{L^4(B_r)}|\eps\om_n|^{1-\frac{1}{4}}\sup_{\eps\om_n}\left|\Delta \tilde{u}-\Delta\tilde{u}(0)\right|\leq C'\eps^{2+\frac{1}{2}}\\
\left|\int_{\eps\om_n}\nabla \tilde{z}\cdot \eps^2\nabla\tilde{v}\right|&\leq C\eps^2\Vert \tilde{z}\Vert_{H^1(B_r)}\Vert \nabla\tilde{v}\Vert_{L^4(B_{r/2})}|\eps\om_n|^{\frac{1}{4}}\leq C'\eps^{2+\frac{1}{2}}\\
\left|\int_{\eps\om_n}\eps^2\tilde{z}\Delta\tilde{v}\right|&\leq C\eps^2\Vert \tilde{z}\Vert_{L^{12}(B_r)}\Vert \Delta \tilde{v}\Vert_{L^{3/2}(B_r)}|\eps\om_n|^\frac{1}{4}\leq C'\eps^{2+\frac{1}{2}}
\end{align*}
Combining the estimates above, we see that the boundary term (near the $n$-th hole) is bounded by $C\eps^{2+\frac{1}{2}}$, as desired for Lemma \ref{lem_estansatz}. 

\subsubsection*{Regular term 1}
We start by considering the first regular term $\int_{\Om^\eps}z(\Delta+\mu_{\eps})u_{\eps}$, splitting it into a bulk term and terms near the holes:
\begin{equation} \label{regularterms}
\int_{\Om^\eps}z(\Delta+\mu_{\eps})u_{\eps}=\int_{\Om^\eps\setminus\cup_n \psi_n(B_r)}z(\Delta +\mu_{\eps})u_{\eps}+\sum_{n=1}^{N}\int_{B_r\setminus\eps\om_n}\tilde{z}(\Delta \tilde{u}_\eps+\rho \mu_\eps \tilde{u}_\eps) .
\end{equation}
We start with the term near the $n$-th hole. Let us give a few preliminary estimates on the exterior problems for $W_n$ and $L_n$ in Section \ref{sec_conv}: there is a constant $C>0$ such that for all $x\in\R^2\setminus\om_n$, 
\[\left|W_n(x)-\frac{M_nx}{2\pi|x|^2}\right|\leq \frac{C}{|x|^2},\qquad \left|\partial_{x_p} W_n(x)-\partial_{x_p}\frac{M_nx}{2\pi|x|^2}\right|\leq \frac{C}{|x|^3},\]
and similarly
\[\left|L_n(x)-\frac{|\om_n|}{2\pi}(\log|x|)I_2\right|\leq \frac{C}{|x|},\qquad \left|\partial_{x_p} L_n(x)-\frac{|\om_n|}{2\pi}(\partial_{x_p}\log|x|)I_2\right|\leq \frac{C}{|x|^2},\]
for $p=1,2$. 
 
Formulas \eqref{def_Fni}, \eqref{vdef} and \eqref{eq_ansatz} give equations satisfied by the functions $F_n, \tilde{v} = v \circ \psi_n$ and $\tilde{u}_\eps=u_\eps \circ \psi_n$ when $x\in B_r\setminus\eps\om_n$, and by using those equations we may compute 

\begin{align*}
& (\Delta+\rho \mu_\eps)\tilde u_\eps(x) \\
&=(\Delta +\rho \mu) \tilde{u}  \\
&\quad +\eps(\Delta+\rho\mu)\left[\zeta(x)\left(W_n(x/\eps)\cdot\nabla\tilde  u(0)- \frac{\langle x/ \eps  ,M_n\nabla \tilde u(0)\rangle}{2\pi|x/ \eps  |^2}\right)\right]\\
&\quad+\eps^2(\Delta+\rho\mu)\left[\zeta(x)\left(- L_n(x/\eps):D^2 \tilde u(0)-\frac{\rho_n(0)\mu|\om_n|\tilde u(0)}{2\pi}\log |x/\eps| \right)\right]\\
&\quad+\eps^3  \rho  \kappa \left[\zeta(x)W_n(x/\eps)\cdot\nabla\tilde{u}(0)+\eps \tilde v(x)+\eps \zeta(x)\left(-L_n(x/\eps):D^2 \tilde u(0)+\frac{\rho_n(0)\mu|\om_n|\tilde u(0)}{2\pi}\log \eps \right)\right] .
\end{align*}
Since $(\Delta +\rho \mu) \tilde{u}=0$ and $W_n(x), L_n(x), x/|x|^2$ and $\log |x|$ are harmonic, the last formula implies that 
\begin{align*}
(\Delta+\rho \mu_\eps)\tilde u_\eps(x) 
&= \eps(\Delta \zeta+  \rho  \mu \zeta)\left(W_n(x/\eps)\cdot\nabla \tilde{u}(0)- \frac{\langle x/ \eps ,M_n\nabla \tilde{u}(0)\rangle}{2\pi|x/ \eps |^2}\right)\\
&\quad+2\sum_{p=1}^{2}\partial_{x_p}\zeta\left( (\partial_{x_p} W_n)(x/\eps)\cdot\nabla \tilde{u}(0) - \partial_{x_p} \! \left. \frac{\langle \, \cdot \, ,M_n\nabla \tilde{u}(0)\rangle}{2\pi|\cdot|^2} \right|_{x/\eps} \right)\\
&\quad-\eps^2(\Delta\zeta+  \rho  \mu\zeta)\left(L_n(x/\eps):D^2 \tilde{u}(0)+\frac{\rho_n(0)\mu|\om_n| \tilde{u}(0)}{2\pi}\log |x/\eps| \right)\\
&\quad-2\eps\sum_{p=1}^{2} \partial_{x_p}\zeta\left( (\partial_{x_p} L_n)(x/\eps):D^2 \tilde{u}(0)+\frac{\rho_n(0)\mu|\om_n| \tilde{u}(0)}{2\pi}  (\partial_{x_p} \log |\cdot|)(x/\eps)  \right)+\mathcal{O}\left(\eps^3\right) .
\end{align*}
Using the error bounds on the exterior problems and noting that $I_2 : D^2 \tilde{u}(0)=\Delta \tilde{u}(0)$, we find 
\[
\left|(\Delta+\rho \mu_\eps)\tilde u_\eps(x)\right|\leq C\left(\eps^3+\frac{\eps^3}{|x|^2}\right)  \leq C \eps^3 \frac{1}{|x|^2}  , \qquad x\in B_r\setminus \eps\om_n .
\]
As a consequence, and using the fact that $\tilde{z}$ is bounded in $L^4(B_r)$, we get that 
\begin{align*}
\int_{B_r\setminus\eps\om_n}\tilde{z}(\Delta \tilde{u}_\eps+\rho \mu_\eps \tilde{u}_\eps)&\leq C\eps^3 \, \Vert \tilde{z}\Vert_{L^4(B_r)} \left(\int_{B_r\setminus\eps\om_n}\!\! \left(\frac{1}{|x|^2}\right)^{\!\! 4/3}\right)^{\!\! 3/4} \\
&\leq C\eps^{2+\frac{1}{2}} .
\end{align*}

Next we treat the bulk term in \eqref{regularterms}, which is $\int_{\Om^\eps\setminus\cup_n \psi_n(B_r)}z(\Delta +\mu_{\eps})u_{\eps}$. First note that $u_\eps$ is bounded in $\mathcal{C}^p(\Om^\eps\setminus \cup_n \psi_n(B_r))$ for any $p\in\N$, by the definition \eqref{eq_ansatz} of $u_\eps$ and elliptic regularity for $v$ (which is valid up to the boundary since $\partial\Om$ is smooth and the Neumann condition is satisfied). 

Denoting by $g_t$ the metric on $\Omega$ obtained by pulling back the area measure of $\phi^{t}(\Om)$, and by $\Delta_t$ and $m_t$ the associated Laplacian and area measure, we have
\begin{align*}
\int_{\Om^\eps\setminus\cup_n \psi_n(B_r)}z(\Delta +\mu_{\eps})u_{\eps}&=\int_{\Om\setminus\cup_n \psi_n(B_r)} (z\circ \phi^{\eps^2}) (\Delta_{\eps^2} +\mu_\eps )\widehat{u}_\eps \, dm_{\eps^2}
\end{align*}
where $\widehat{u}_\eps =u_\eps\circ\phi^{\eps^2}$ was defined in the ansatz \eqref{eq_ansatz}. Then, in $\Om\setminus\cup_n \psi_n(B_r)$:
\begin{align*}
(\Delta_{\eps^2} +\mu_\eps )\widehat{u}_\eps &=(\Delta +\mu+\eps^2( K+\kappa)+(\Delta_{\eps^2}-\Delta-\eps^2 K))(u+\eps^2v) \qquad \text{by ansatz \eqref{eq_ansatz}} \\
&=(\Delta+\mu)u+\eps^2\left(\Delta v+\mu v+\kappa u+Ku\right)\\
&\quad +(\Delta_{\eps^2}-\Delta-\eps^2 K)(u+\eps^2v) +\eps^4  (K+\kappa)  v .
\end{align*}
The first of these two lines vanishes by the definitions of $u$ and $v$, and the second line is of order $\mathcal{O}(\eps^4)$ (by definition of $K$, $\Delta_{\eps^2}-\Delta-\eps^2 K$ is a second order differential operator with coefficient bounded uniformly by $C\eps^4$).

\subsubsection*{Regular term 2}
Finally, we treat similarly the second regular term of equation \eqref{eq_decompositionregularboundary}, which is $\int_{\partial\phi^{\eps^2}(\Om)}z \, \partial_\nu u_\eps$. 

Recall from earlier in the paper that $\nu_t=(\phi^t)^*\nu_{\partial\phi^t(\Om)}$ is the outward normal vector with respect to the metric $g_t$ on $\Om$. Taking $t=\eps^2$, the formula says $\nu_{\eps^2}=(\phi^{\eps^2})^*\nu$ where $\nu$ is the outward normal vector on $\partial\phi^{\eps^2}(\Om)$. Hence the normal derivative of $u_\eps$ is $\partial_\nu u_\eps = du_\eps(\nu)=d\widehat{u}_\eps(\nu_{\eps^2})$, using that $u_\eps \circ \phi^{\eps^2} = \widehat{u}_\eps$. Thus 

\begin{align*}
\partial_\nu u_\eps = d\widehat{u}_\eps(\nu_{\eps^2})&=\langle\nabla\widehat{u}_\eps,\nu_{\eps^2}\rangle\\
&=\langle \nabla(u+\eps^2 v),\nu +\eps^2\eta +(\nu_{\eps^2}-\nu-\eps^2 \eta)\rangle\\
&=\langle\nabla u,\nu\rangle+\eps^2\left(\langle \nabla v,\nu\rangle+\langle \nabla u,\eta\rangle\right)\\
&\quad +\langle \nabla(u+\eps^2 v),\nu_{\eps^2}-\nu-\eps^2 \eta\rangle+\eps^4 \langle \nabla v,\eta\rangle .
\end{align*}
The first line vanishes by the Neumann boundary conditions on $u$ and $v$ (that is, $\langle\nabla u,\nu\rangle=0$ and $\langle \nabla v,\nu\rangle=-\langle \nabla u,\eta\rangle$), and the second line is bounded by $\mathcal{O}(\eps^4)$. 

This uniform estimate along with a trace inequality for $z$ on $\partial \Om^\eps$ completes the proof that the second regular term of equation \eqref{eq_decompositionregularboundary} is bounded by $C\eps^4$.

\begin{lemma}[Harmonic extension across a hole] \label{lem_extension_holes}
Let $\om\subset\R^2$ be a smoothly bounded open set. If $h\in H^1(\R^2)$ is compactly supported and harmonic on $\om$ then $\int_{\om}|\nabla h|^2\leq C_\om \int_{\R^2\setminus\om}|\nabla h|^2$ for some constant $C_\om>0$.
\end{lemma}
\begin{proof}
Without loss of generality, suppose $\om$ contains the origin. Write $i(z)=z/|z|^2$ for the inversion map on $\R^2\cup\{\infty\}$, so that $\om^*=i(\R^2\cup\{\infty\}\setminus\overline{\om})$ is also a smoothly bounded domain that contains the origin. Denote the fractional semi-norm by $[h]_{H^{1/2}(\partial \om)}=\left( \int_{\partial\om} \int_{\partial\om} \frac{|h(x)-h(y)|^2}{|x-y|^2} \, ds(x)ds(y) \right)^{\! 1/2}$ where $ds$ is the arclength element. Then for some constants $C_1,C_2,C_3>0$, we have 
\[
\int_{\om}|\nabla h|^2\leq C_1[h]_{H^{1/2}(\partial\om)}^2\leq C_2[h\circ i]_{H^{1/2}(\partial\om^*)}^2\leq C_3\int_{\om^*}|\nabla(h\circ i)|^2=C_3\int_{\R^2\setminus\om}|\nabla h|^2
\]
where the first ``inverse trace'' inequality comes from harmonicity of $h$, the second inequality uses the definition of the semi-norm, the third is the  $H^1\to H^{1/2}$ trace theorem applied to $u-\text{avg}_\Om(u)$ and followed by Poincar\'{e}'s inequality, and the final equality relies on invariance of the Dirichlet integral under the conformal map $i$. 
\end{proof}

\section*{Acknowledgements} 
Laugesen’s research was supported by grants from the Simons Foundation (964018) and the U.S. National Science Foundation (2246537).

\bibliographystyle{plain}

\end{document}